\documentclass[reqno]{amsart}%
\usepackage{amssymb,mathtools,mathrsfs}

\numberwithin{equation}{section}

\usepackage{ifpdf}
\ifpdf
 \usepackage[hyperindex]{hyperref}
\else
 \expandafter\ifx\csname dvipdfm\endcsname\relax
 \usepackage[hypertex,hyperindex]{hyperref}%
 \else
 \usepackage[dvipdfm,hyperindex]{hyperref}%
 \fi
\fi
\allowdisplaybreaks[4]

\theoremstyle{plain}
\newtheorem*{problem}{Qi's problem}
\newtheorem{thm}{Theorem}[section]
\newtheorem{lem}{Lemma}[section]
\newtheorem{cor}{Corollary}[section]

\theoremstyle{remark}
\newtheorem{rem}{Remark}[section]

\DeclareMathOperator{\te}{e}

\begin{document}

\title[Monotonicity for ratios of Bernoulli polynomials]
{Monotonicity and inequalities for ratios of Bernoulli polynomials and numbers}

\author[Z.-H. Yang]{Zhen-Hang Yang}
\address{Department of Science and Technology, Research Institute of State Grid Zhejiang Electric Power Company, Hangzhou, Zhejiang, 310014, China}
\email{yzhkm@163.com}
\urladdr{\url{https://orcid.org/0000-0002-2719-4728}}

\author[F. Qi]{Feng Qi*}
\address{School of Mathematics and Physics, Hulunbuir University, Hailar, Inner Mongolia, 021008, China;
17709 Sabal Court, University Village, Dallas, TX 75252-8024, USA}
\email{\href{mailto: F. Qi <qifeng618@gmail.com>}{qifeng618@gmail.com}}
\urladdr{\url{https://orcid.org/0000-0001-6239-2968}}

\begin{abstract}
In the work, the authors establish the monotonicity of the ratios
$$
\frac{B_{2\ell-1}(s)}{B_{2\ell+1}(s)}, \quad
\frac{B_{2\ell}(s)}{B_{2\ell+1}(s)}, \quad
\frac{B_{2m}(s)}{B_{2\ell}(s)}, \quad
\frac{B_{2\ell}(s)}{B_{2\ell-1}(s)},
$$
where $B_\ell(s)$ denotes the Bernoulli polynomials. Using these newly established monotonicity properties, the authors derive several new inequalities and also recover a number of known inequalities involving the Bernoulli polynomials $B_\ell(s)$, the Bernoulli numbers $B_{2\ell}$, and their ratios such as $\frac{B_{2\ell+2}}{B_{2\ell}}$.
\end{abstract}



\thanks{*Corresponding author: Feng Qi, qifeng618@gmail.com}

\thanks{This paper was typeset using \AmS-\LaTeX}

\maketitle

\section{Motivations}

Throughout this paper, we use the notation $\mathbb{N}=\{1,2,\dotsc\}$ for the set of positive integers and $\mathbb{N}_0=\{0\}\cup\mathbb{N}$ for the set of nonnegative integers.
\par
It is well known~\cite[p.~3]{Temme-96-book} that the classical Bernoulli polynomials $B_\ell(s)$ for $\ell\in\mathbb{N}_0$ are generated by
\begin{equation*}
\frac{z\te^{s z}}{\te^z-1}=\sum_{\ell=0}^{\infty}B_\ell(s)\frac{z^\ell}{\ell!}, \quad |z|<2\pi
\end{equation*}
and $B_\ell(0)=B_\ell$ are the Bernoulli numbers, which satisfy
$$
B_0=1,\quad
B_1=-\frac12,\quad
B_{2\ell}\neq0,\quad
B_{2\ell+1}=0,
$$
for all $\ell\in\mathbb{N}$.
\par
On 2 January 2024, Qi posed on MathOverflow (\url{https://mathoverflow.net/q/461427}) the problem of investigating the monotonicity of the ratio $\frac{B_{2\ell+1}(s)}{B_{2\ell+3}(s)}$ for $\ell\in\mathbb{N}_0$.

\begin{problem}
The ratio $\bigl|\frac{B_{2\ell+1}(s)}{B_{2\ell+3}(s)}\bigr|$ for $\ell\in\mathbb{N}_0$ is decreasing in $s\in\bigl(0,\frac{1}{2}\bigr)$ and increasing in $s\in\bigl(\frac12,1\bigr)$.
\end{problem}
\par
On 3 January 2024, Yang communicated with Qi via WeChat and provided a solution to Qi's problem, which is essentially the original version of the proof presented in~\cite[Proposition~1]{Bernouli-No-Tail.tex}.

\begin{thm}[{\cite[Proposition~1]{Bernouli-No-Tail.tex}}]\label{bernou-ratio-assum}
The ratio $\frac{B_{2\ell-1}(s)}{B_{2\ell+1}(s)}$ for $\ell\in\mathbb{N}$ increases in $s\in\bigl(0,\frac{1}{2}\bigr)$ and decreases in $s\in\bigl(\frac{1}{2},1\bigr)$.
\end{thm}

On 4 January 2024, Iosif Pinelis provided an alternative solution to Qi's problem, together with some further results, at \url{https://mathoverflow.net/a/461546}. His proof was subsequently published in~\cite{Pinelis-Lie-MIA-24}.
\par
In Section~\ref{Qi-probl-3rd-proof-sec}, we present the third solution to Qi's problem, namely, the third proof of Theorem~\ref{bernou-ratio-assum}. In Section~\ref{monot-two-bern-polyn-sec}, we establish the monotonicity of three additional ratios
$\frac{B_{2\ell}(s)}{B_{2\ell+1}(s)}$, $\frac{B_{2m}(s)}{B_{2\ell}(s)}$, and $\frac{B_{2\ell}(s)}{B_{2\ell-1}(s)}$.
In Section~\ref{inequalities-sec-bernoulli}, we derive several new inequalities and recover some known inequalities for the Bernoulli polynomials $B_\ell(s)$, the Bernoulli numbers $B_{2\ell}$, and their ratios, such as $\frac{B_{2\ell+2}}{B_{2\ell}}$.

\section{Lemmas and properties of Bernoulli polynomials}

To advance our objectives, we first recall two useful lemmas together with essential properties of the Bernoulli polynomials $B_\ell(s)$.

\subsection{Two lemmas}
The following two lemmas are useful in Section~\ref{Qi-probl-3rd-proof-sec} of this paper.

\begin{lem}[{\cite[pp.~10--11, Theorem~1.25]{AVV-1997}}]\label{AVV-1997-mon-rule}
For $\alpha,\beta\in\mathbb{R}$ with $\alpha<\beta$, let $q(s)$ and $p(s)$ be continuous on $[\alpha,\beta]$, differentiable on $(\alpha,\beta)$, and $p'(s)\ne 0$ on $(\alpha,\beta)$.
If the ratio $\frac{q'(s)}{p'(s)}$ increases in $s\in(\alpha,\beta)$, then both $\frac{q(s)-q(\alpha)}{p(s)-p(\alpha)}$ and $\frac{q(s)-q(\beta)}{p(s)-p(\beta)}$ increase in $s\in(\alpha,\beta)$.
\end{lem}

\begin{lem}[\cite{Ostrowski-1960-Murch}]\label{Ostrowski-1960-Murch-lem}
The Bernoulli polynomial $B_{2\ell}(s)$ for $\ell\in\mathbb{N}$ has one real zero, denoted by $r_{2\ell}$, on the interval $\bigl(0,\frac{1}{2}\bigr)$. The sequence $r_{2\ell}$ increases in $\ell$ and tends to $\frac{1}{4}$ as $\ell\to\infty$.
\end{lem}

\begin{rem}
In~\cite{Norlund-1922}, the double inequality $\frac{1}{6}<r_{2\ell}<\frac{1}{4}$ was established for $\ell\in\mathbb{N}$. In~\cite[p.~534]{Lehmer-Monthly-1940}, a more precise estimate $\frac{1}{4}-\frac{1}{2^{2\ell+1}\pi}<r_{2\ell}<\frac{1}{4}$ was presented for $\ell\in\mathbb{N}$. See also~\cite[p.~594, Section~24.12(i)]{NIST-HB-2010}.
\end{rem}

\subsection{Some properties of Bernoulli polynomials}
The Bernoulli polynomials $B_\ell(s)$ have the following properties:
\begin{enumerate}
\item
For $0<s<\frac{1}{2}$, the positivity
\begin{equation}\label{Bern-polyn-half-posit}
(-1)^{\ell+1}B_{2\ell+1}(s)>0, \quad \ell\in\mathbb{N}_0
\end{equation}
and
\begin{equation}\label{Even-Bernoulli-posit}
(-1)^{\ell+1}B_{2\ell}>0, \quad \ell\in\mathbb{N}
\end{equation}
hold. See~\cite[p.~805, Entry~23.1.14]{abram} and~\cite[p.~588, Entry~24.2.2]{NIST-HB-2010}.
\item
For $\ell\in\mathbb{N}$, the recursive relation
\begin{equation}\label{Bernou-polyn-deriv}
B_\ell'(s)=\ell B_{\ell-1}(s)
\end{equation}
validates. See~\cite[p.~590, Entry~24.4.34]{NIST-HB-2010}.
\item
For $\ell\in\mathbb{N}_0$, the identity
\begin{equation}\label{Bern-polyn-symm-ID}
B_{\ell}(1-s)=(-1)^\ell B_{\ell}(s)
\end{equation}
validates. See~\cite[p.~804, Entry~23.1.8]{abram}.
\item
For $\ell\in\mathbb{N}_0$, the relation
\begin{equation}\label{Bernoulli-polyn-numb}
B_\ell\biggl(\frac{1}{2}\biggr)=-\biggl(1-\frac{1}{2^{\ell-1}}\biggr)B_\ell
\end{equation}
holds. See~\cite[p.~590, Entry~24.4.27]{NIST-HB-2010}.
\item
For $\ell\in\mathbb{N}$, the relation
\begin{equation}\label{abramp806-nistp590}
B_\ell\biggl(\frac{1}{4}\biggr)=(-1)^\ell B_\ell\biggl(\frac{3}{4}\biggr)=-\frac{1-2^{1-\ell}}{2^\ell}B_\ell-\frac{\ell}{4^\ell}E_{\ell-1}
\end{equation}
validates, where $E_\ell$ for $\ell\in\mathbb{N}_0$ denotes the Euler numbers generated by
\begin{equation*}
\frac{2\te^z}{\te^{2z}+1}
=\sum_{\ell=0}^\infty E_\ell\frac{z^\ell}{\ell!}
=\sum_{\ell=0}^\infty E_{2\ell}\frac{z^{2\ell}}{(2\ell)!}, \quad |z|<\frac\pi2.
\end{equation*}
See~\cite[p.~806, Entry~23.1.22]{abram}, \cite[p.~534]{Lehmer-Monthly-1940}, and~\cite[p.~590, Entry~24.4.31]{NIST-HB-2010}.
\item
As $\ell\to\infty$, the asymptotic approximations
\begin{equation}\label{p594Entry24.11.5NIST-HB-2010}
(-1)^{\lfloor \ell/2\rfloor-1}\frac{(2\pi)^\ell}{2(\ell!)}B_\ell(s)\sim
\begin{cases}
\cos(2\pi s), & \ell=2k\\
\sin(2\pi s), & \ell=2k-1
\end{cases}
\end{equation}
validate uniformly for $s$ on compact subsets of $\mathbb{C}$, where $\lfloor \ell/2\rfloor$ denotes the floor function whose value is the largest integer less than or equal to $\frac{\ell}{2}$. See~\cite[p.~594, Entry~24.11.5]{NIST-HB-2010}.
\item
Let $\zeta(z)$ denote the Riemann zeta function which is defined by the series $\zeta(z)=\sum_{\ell=1}^\infty\frac1{\ell^z}$ under the condition $\Re(z)>1$ and by analytic continuation elsewhere. The Bernoulli numbers $B_{2\ell}$ and the Riemmann zeta function $\zeta(z)$ have the relation
\begin{equation}\label{Bernoul-zeta-rel}
B_{2\ell}=\frac{(-1)^{\ell+1}2(2\ell)!}{(2\pi)^{2\ell}}\zeta(2\ell), \quad \ell\in\mathbb{N}.
\end{equation}
See~\cite[pp.~807--808, Section~23.2]{abram} and~\cite[p.~5, (1.14)]{Temme-96-book}.
\end{enumerate}

\section{Third solution to Qi's problem}\label{Qi-probl-3rd-proof-sec}

In this section, we present the third solution to Qi's problem, namely the third proof of Theorem~\ref{bernou-ratio-assum}. In fact, this third proof of Theorem~\ref{bernou-ratio-assum} was written down slightly earlier than the original version of the proof of~\cite[Proposition~1]{Bernouli-No-Tail.tex}.

\begin{proof}[Third proof of Theorem~\ref{bernou-ratio-assum}]
Define the function
\begin{equation}\label{H(fg)(s)-dfn}
H_{f,g}(s)=\frac{f'(s)}{g'(s)}g(s)-f(s), \quad s\in(a,b),
\end{equation}
where $f(s)$ and $g(s)$ are two differentiable functions on $(a,b)$ such that $g'(s)\ne0$ on $(a,b)$. If $f(s)$ and $g(s)$ are twice differentiable on $(a,b)$, then
\begin{equation}\label{H(fg)(s)fg-dfn}
\biggl[\frac{f(s)}{g(s)}\biggr]'=\frac{g'(s)}{g^2(s)}\biggl[\frac{f'(s)}{g'(s)}g(s)-f(s)\biggr]
=\frac{g'(s)}{g^2(s)}H_{f,g}(s)
\end{equation}
and
\begin{equation*}
H_{f,g}'(s)=\biggl[\frac{f'(s)}{g'(s)}\biggr]'g(s).
\end{equation*}
The function $H_{f,g}(s)$ in~\eqref{H(fg)(s)-dfn} was first introduced in the preprint~\cite{Yang-arXiv-2014}, it was called Yang's $H$-function in~\cite{Tian-AIMS-M-5-2020}, and it has been applied, for example, in~\cite{Alzer-ITSF-31-2020, Richards-ITSF-32-2020, Tian-RM-78-2023, Yang-JMAA-508-2022, Yang-JMAA-470-2019} and~\cite[Theorem 2.1]{Yang-JMAA-428-2015}.
\par
We write
\begin{equation}\label{Bernoul-phi-rel-ratio}
\frac{B_{2\ell-1}(s)}{B_{2\ell+1}(s)}=-\frac{(-1)^\ell B_{2\ell-1}(s)}{(-1)^{\ell+1}B_{2\ell+1}(s)}
=-\frac{\phi_\ell(s)}{\phi_{\ell+1}(s)}, \quad \ell\in\mathbb{N},
\end{equation}
where, by the positivity in~\eqref{Bern-polyn-half-posit},
\begin{equation*}
\phi_\ell(s)=(-1)^\ell B_{2\ell-1}(s)>0, \quad \ell\in\mathbb{N}
\end{equation*}
is clearly valid in $s\in\bigl(0,\frac{1}{2}\bigr)$.
\par
By the formula~\eqref{Bernou-polyn-deriv}, we obtain
\begin{equation}\label{phi-deriv-first}
\phi_\ell'(s)=(-1)^\ell B_{2\ell-1}'(s)=(-1)^\ell(2\ell-1)B_{2\ell-2}(s), \quad \ell\in\mathbb{N}
\end{equation}
and
\begin{multline}\label{phi-deriv-rela}
\phi_\ell''(s)=(-1)^\ell B_{2\ell-1}''(s)
=(-1)^\ell(2\ell-1)(2\ell-2)B_{2\ell-3}(s)\\
=-(2\ell-1)(2\ell-2)\phi_{\ell-1}(s), \quad \ell\ge2.
\end{multline}
Accordingly, the second derivative $\phi_\ell''(s)$ for $\ell\ge2$ is negative in $s\in\bigl(0,\frac{1}{2}\bigr)$, the first derivative $\phi_\ell'(s)$ for $\ell\ge2$ decreases in $s\in\bigl(0,\frac{1}{2}\bigr)$, and the function $\phi_\ell(s)$ for $\ell\ge2$ is concave in $s\in\bigl(0,\frac{1}{2}\bigr)$.
\par
From the identity~\eqref{Bernoulli-polyn-numb}, it follows that
\begin{equation*}
\phi_\ell'(0)=(-1)^\ell(2\ell-1)B_{2\ell-2}(0)=(-1)^\ell(2\ell-1)B_{2\ell-2}>0
\end{equation*}
and
\begin{equation*}
\phi_\ell'\biggl(\frac{1}{2}\biggr)=(-1)^\ell(2\ell-1)B_{2\ell-2}\biggl(\frac{1}{2}\biggr)
=(-1)^{\ell+1}\biggl(1-\frac{1}{2^{2\ell-3}}\biggr)(2\ell-1)B_{2\ell-2}
<0
\end{equation*}
for $\ell\ge2$, where we used the positivity in~\eqref{Even-Bernoulli-posit}.
As a result, the first derivative $\phi_\ell'(s)$ has a unique zero $s_{\ell,0}\in\bigl(0,\frac{1}{2}\bigr)$ for $\ell\ge2$.
\par
As done in the proof of~\cite[Proposition~1]{Bernouli-No-Tail.tex}, both $\frac{B_1(s)}{B_3(s)}$ and $\frac{B_3(s)}{B_5(s)}$ increase in $s\in\bigl(0,\frac{1}{2}\bigr)$. Equivalently, both $\frac{\phi_1(s)}{\phi_{2}(s)}$ and $\frac{\phi_2(s)}{\phi_{3}(s)}$ decrease in $s\in\bigl(0,\frac{1}{2}\bigr)$.
\par
Assume that, for some $\ell\ge3$, the ratio $\frac{\phi_\ell(s)}{\phi_{\ell+1}(s)}$ decreases in $s\in\bigl(0,\frac{1}{2}\bigr)$. By~\eqref{phi-deriv-rela}, this inductive hypothesis is equivalent to the decreasing property of the ratio
\begin{equation*}
\frac{\phi_{\ell+1}''(s)}{\phi_{\ell+2}''(s)}=\frac{(2\ell+1)(2\ell)}{(2\ell+1)(2\ell+2)} \frac{\phi_\ell(s)}{\phi_{\ell+1}(s)}, \quad \ell\ge3
\end{equation*}
in $s\in\bigl(0,\frac{1}{2}\bigr)$, that is,
\begin{equation*}
\biggl[\frac{\phi_{\ell+1}''(s)}{\phi_{\ell+2}''(s)}\biggr]'<0, \quad \ell\ge3, \quad s\in\biggl(0,\frac{1}{2}\biggr).
\end{equation*}
This means that
\begin{equation*}
H_{\phi_{\ell+1}'(s),\phi_{\ell+2}'(s)}'(s)=\biggl[\frac{\phi_{\ell+1}''(s)}{\phi_{\ell+2}''(s)}\biggr]'\phi_{\ell+2}'(s), \quad \ell\ge3
\end{equation*}
has a unique zero $s_{\ell+2,0}\in\bigl(0,\frac{1}{2}\bigr)$ and that
\begin{multline*}
H_{\phi_{\ell+1}'(s),\phi_{\ell+2}'(s)}(s)\ge H_{\phi_{\ell+1}'(s),\phi_{\ell+2}'(s)}(s_{\ell+2,0})\\
=\frac{\phi_{\ell+1}''(s_{\ell+2,0})}{\phi_{\ell+2}''(s_{\ell+2,0})}\phi_{\ell+2}'(s_{\ell+2,0})-\phi_{\ell+1}'(s_{\ell+2,0})
=-\phi_{\ell+1}'(s_{\ell+2,0})
\end{multline*}
for some $\ell\ge3$ and $s\in\bigl(0,\frac{1}{2}\bigr)$.
\par
From the equation~\eqref{phi-deriv-first} and Lemma~\ref{Ostrowski-1960-Murch-lem}, we see that the zero $s_{\ell+2,0}$ is just the zero of the Bernoulli polynomial $B_{2\ell+2}(s)$, that is, $s_{\ell+2,0}=r_{2\ell+2}$, for some $\ell\ge3$. Similarly, the function $\phi_{\ell+1}'(s)$ has a unique zero $r_{2\ell}=s_{\ell+1,0}\in\bigl(0,\frac{1}{2}\bigr)$ for some $\ell\ge3$. From Lemma~\ref{Ostrowski-1960-Murch-lem}, it follows that $r_{2\ell}=s_{\ell+1,0}<s_{\ell+2,0}=r_{2\ell+2}$ for some $\ell\ge3$. Since the first derivative $\phi_\ell'(s)$ for $\ell\ge2$ decreases in $s\in\bigl(0,\frac{1}{2}\bigr)$, we acquire that
\begin{equation*}
\phi_{\ell+1}'(s_{\ell+2,0})<\phi_{\ell+1}'(s_{\ell+1,0})=0, \quad \ell\ge2.
\end{equation*}
Consequently, we conclude $H_{\phi_{\ell+1}'(s),\phi_{\ell+2}'(s)}(s)>0$ for some $\ell\ge3$ and $s\in\bigl(0,\frac{1}{2}\bigr)$. This conclusion means that
\begin{equation*}
\biggl[\frac{\phi_{\ell+1}'(s)}{\phi_{\ell+2}'(s)}\biggr]'
=\frac{\phi_{\ell+2}''(s)}{[\phi_{\ell+2}'(s)]^2}H_{\phi_{\ell+1}'(s),\phi_{\ell+2}'(s)}(s)<0
\end{equation*}
for some $\ell\ge3$ and $s\in\bigl(0,\frac{1}{2}\bigr)\setminus\{s_{\ell+2,0}\}$, where we used the negativity $\phi_\ell''(s)<0$ for $\ell\ge2$ and $s\in\bigl(0,\frac{1}{2}\bigr)$. Thus, the ratio $\frac{\phi_{\ell+1}'(s)}{\phi_{\ell+2}'(s)}$ for some $\ell\ge3$ decreases on the intervals $(0,r_{2\ell+2})$ and $\bigl(r_{2\ell+2},\frac{1}{2}\bigr)$, respectively.
\par
Since
\begin{equation*}
\phi_\ell(0)=(-1)^\ell B_{2\ell-1}(0)=(-1)^\ell B_{2\ell-1}=0
\end{equation*}
for $\ell\ge2$, by virtue of Lemma~\ref{AVV-1997-mon-rule} and the decreasing property of the ratio $\frac{\phi_{\ell+1}'(s)}{\phi_{\ell+2}'(s)}$ for some $\ell\ge3$ on the interval $(0,r_{2\ell+2})$, we reveal that the ratio
\begin{equation*}
\frac{\phi_{\ell+1}(s)}{\phi_{\ell+2}(s)}=\frac{\phi_{\ell+1}(s)-\phi_{\ell+1}(0)}{\phi_{\ell+2}(s)-\phi_{\ell+2}(0)}
\end{equation*}
for some $\ell\ge3$ decreases in $s\in(0,r_{2\ell+2})$.
By the identity~\eqref{Bern-polyn-symm-ID}, it follows that
\begin{equation*}
\phi_\ell\biggl(\frac{1}{2}\biggr)=(-1)^\ell B_{2\ell-1}\biggl(\frac{1}{2}\biggr)=0
\end{equation*}
for $\ell\ge2$. In view of Lemma~\ref{AVV-1997-mon-rule} and the decreasing property of the ratio $\frac{\phi_{\ell+1}'(s)}{\phi_{\ell+2}'(s)}$ for some $\ell\ge3$ on the interval $\bigl(r_{2\ell+2},\frac{1}{2}\bigr)$, we arrive at that the ratio
\begin{equation*}
\frac{\phi_{\ell+1}(s)}{\phi_{\ell+2}(s)}=\frac{\phi_{\ell+1}(s)-\phi_{\ell+1}\bigl(\frac{1}{2}\bigr)}{\phi_{\ell+2}(s)-\phi_{\ell+2}\bigl(\frac{1}{2}\bigr)}
\end{equation*}
for some $\ell\ge3$ decreases in $\bigl(r_{2\ell+2},\frac{1}{2}\bigr)$. Since the ratio $\frac{\phi_{\ell+1}(s)}{\phi_{\ell+2}(s)}$ is continuous at the point $r_{2\ell+2}=s_{\ell+2,0}$, which is the zero of $\phi_{\ell+2}'(s)$ on $\bigl(0,\frac{1}{2}\bigr)$, we conclude that the ratio $\frac{\phi_{\ell+1}(s)}{\phi_{\ell+2}(s)}$ decreases in $s\in\bigl(0,\frac{1}{2}\bigr)$ for some $\ell\ge3$.
\par
By induction, the ratio $\frac{\phi_\ell(s)}{\phi_{\ell+1}(s)}$ decreases in $s\in\bigl(0,\frac{1}{2}\bigr)$ for all $\ell\in\mathbb{N}$. Therefore, from the relation~\eqref{Bernoul-phi-rel-ratio}, the ratio $\frac{B_{2\ell-1}(s)}{B_{2\ell+1}(s)}$ for $\ell\in\mathbb{N}$ increases in $s\in\bigl(0,\frac{1}{2}\bigr)$.
\par
By the identity~\eqref{Bern-polyn-symm-ID}, we obtain
\begin{equation*}
\frac{B_{2\ell-1}(s)}{B_{2\ell+1}(s)}=\frac{B_{2\ell-1}(1-s)}{B_{2\ell+1}(1-s)}
\end{equation*}
for $\ell\in\mathbb{N}$.
Consequently, from the increasing property of the ratio $\frac{B_{2\ell-1}(s)}{B_{2\ell+1}(s)}$ for $\ell\in\mathbb{N}$ in $s\in\bigl(0,\frac{1}{2}\bigr)$, we derive that the ratio $\frac{B_{2\ell-1}(s)}{B_{2\ell+1}(s)}$ for $\ell\in\mathbb{N}$ decreases in $s\in\bigl(\frac{1}{2},1\bigr)$.
The third proof of Theorem~\ref{bernou-ratio-assum} is complete.
\end{proof}

\begin{cor}\label{B(2m-1)(2n-1)-cor}
For $m,\ell\in\mathbb{N}$ such that $m<\ell$, the ratio $(-1)^{\ell-m}\frac{B_{2m-1}(s)}{B_{2\ell-1}(s)}$ is positive and decreases in $s\in\bigl(0,\frac{1}{2}\bigr)$, while it is positive and increases in $s\in\bigl(\frac{1}{2},1\bigr)$.
\end{cor}

\begin{proof}
From~\eqref{Bern-polyn-half-posit} and~\eqref{Bern-polyn-symm-ID}, it follows that the ratio $\frac{B_{2\ell-1}(s)}{B_{2\ell+1}(s)}$ for $\ell\in\mathbb{N}$ is negative in $s\in(0,1)$. Hence, making use of Theorem~\ref{bernou-ratio-assum}, the function $-\frac{B_{2\ell-1}(s)}{B_{2\ell+1}(s)}$ for $\ell\in\mathbb{N}$ is positive and decreases in $s\in\bigl(0,\frac{1}{2}\bigr)$, while it is positive and increases in $s\in\bigl(\frac{1}{2},1\bigr)$. Consequently, the ratio
\begin{equation*}
(-1)^{\ell-m}\frac{B_{2m-1}(s)}{B_{2\ell-1}(s)}
=\prod_{q=m}^{\ell-1}\biggl[-\frac{B_{2q-1}(s)}{B_{2q+1}(s)}\biggr]
\end{equation*}
is positive and decreases in $s\in\bigl(0,\frac{1}{2}\bigr)$, while it is positive and increases in $s\in\bigl(\frac{1}{2},1\bigr)$. The required proof is complete.
\end{proof}

\begin{cor}\label{T4}
For $\ell>m\in\mathbb{N}$, both of the functions
\begin{equation*}
(-1)^{\ell-m}\frac{B_{2m}(s)-B_{2m}}{B_{2\ell}(s)-B_{2\ell}}\quad\text{and}\quad
(-1)^{\ell-m}\frac{B_{2m}(s) -B_{2m}\bigl(\frac12)}{B_{2\ell}(s)-B_{2\ell}\bigl(\frac12\bigr)}
\end{equation*}%
decrease in $s\in\bigl(0,\frac12\bigr)$ and increase in $s\in\bigl(\frac12,1\bigr)$.
\end{cor}

\begin{proof}
This follows from applying Lemma~\ref{AVV-1997-mon-rule}, the identity~\eqref{Bern-polyn-symm-ID}, and Corollary~\ref{B(2m-1)(2n-1)-cor}.
\end{proof}

\section{Monotonicity of more ratios between Bernoulli polynomials}\label{monot-two-bern-polyn-sec}

In this section, we consider three additional ratios
\[
\frac{B_{2\ell}(s)}{B_{2\ell+1}(s)}, \quad 
\frac{B_{2m}(s)}{B_{2\ell}(s)}, \quad 
\frac{B_{2\ell}(s)}{B_{2\ell-1}(s)}
\]
and investigate their monotonicity on the intervals $\bigl(0,\tfrac12\bigr)$ and $\bigl(\tfrac12,1\bigr)$, respectively.

\begin{thm}\label{T5}
The ratio $\frac{B_{2\ell}(s)}{B_{2\ell+1}(s)}$ for $\ell\in\mathbb{N}_0$ decreases in $s\in\bigl(0,\frac12\bigr)$ and in $s\in\bigl(\frac{1}{2},1\bigr)$.
\par
For fixed $s\in\bigl(0,\frac12\bigr)$ \textup{(}or for fixed $s\in\bigl(\frac{1}{2},1\bigr)$, respectively\textup{)}, the sequence $\frac{(2\ell+1)B_{2\ell}(s)}{B_{2\ell+1}(s)}$ increases \textup{(}or decreases, respectively\textup{)} in $\ell\in\mathbb{N}_0$, with the limit
\begin{equation}\label{bernoulli-ratio-n2infty}
\lim_{\ell\to\infty}\frac{(2\ell+1)B_{2\ell}(s)}{B_{2\ell+1}(s)}=2\pi \cot (2\pi s), \quad s\in(0,1)\setminus\biggl\{\frac{1}{2}\biggr\}.
\end{equation}
\end{thm}

\begin{proof}
Let
\begin{equation}\label{f2n=f2n+1}
f_{2\ell}(s)=(-1)^{\ell+1}\frac{B_{2\ell}(s)}{(2\ell)!}\quad\text{and}\quad f_{2\ell+1}(s)=(-1)^{\ell+1}\frac{B_{2\ell+1}(s)}{(2\ell+1)!}
\end{equation}
for $\ell\in\mathbb{N}_0$.
Employing the recursive relation~\eqref{Bernou-polyn-deriv}, we derive
\begin{equation}\label{df2n-2n+1}
f_{2\ell+1}'(s)=f_{2\ell}(s)\quad\text{and}\quad f_{2\ell}'(s)=-f_{2\ell-1}(s).
\end{equation}
From~\eqref{Bern-polyn-half-posit}, we see that the function $f_{2\ell+1}(s)$ is positive for $\ell\in\mathbb{N}_0$ and $0<s<\frac{1}{2}$. Utilizing the relations in~\eqref{df2n-2n+1} and differentiating yield
\begin{equation*}
\biggl[\frac{f_{2\ell}(s)}{f_{2\ell+1}(s)}\biggr]'=\frac{-f_{2\ell-1}(s)f_{2\ell+1}(s)-f_{2\ell}^2(s)}{f_{2\ell+1}^2(s)}<0
\end{equation*}
for $s\in\bigl(0,\frac12\bigr)$ and $\ell\in\mathbb{N}$. Accordingly, the ratio
\begin{equation*}
\frac{f_{2\ell}(s)}{f_{2\ell+1}(s)}=\frac{(2\ell+1)B_{2\ell}(s)}{B_{2\ell+1}(s)}, \quad \ell\in\mathbb{N}
\end{equation*}
decreases in $s\in\bigl(0,\frac12\bigr)$.
In particular, it is obvious that the ratio
\begin{equation*}
\frac{f_0(s)}{f_1(s)}=\frac{B_0(s)}{B_1(s)}=\frac{2}{2s-1}
\end{equation*}
decreases in $s\in\bigl(0,\frac12\bigr)$. Hence, the ratio $\frac{B_{2\ell}(s)}{B_{2\ell+1}(s)}$ for $\ell\in\mathbb{N}_0$ decreases in $s\in\bigl(0,\frac12\bigr)$.
\par
Making use of the identity~\eqref{Bern-polyn-symm-ID} gives
\begin{equation*}
\frac{B_{2\ell}(s)}{B_{2\ell+1}(s)}=-\frac{B_{2\ell}(1-s)}{B_{2\ell+1}(1-s)}
\end{equation*}
for $\ell\in\mathbb{N}_0$ and $s\in\bigl(\frac{1}{2},1\bigr)$. As a result, the ratio $\frac{B_{2\ell}(s)}{B_{2\ell+1}(s)}$ for $\ell\in\mathbb{N}_0$ decreases in $s\in\bigl(\frac{1}{2},1\bigr)$.
\par
For $m,\ell\in\mathbb{N}_{0}$ with $m<\ell$, it is easy to see that
\begin{equation*}
\frac{f_{2m}(s)}{f_{2m+1}(s)}-\frac{f_{2\ell}(s)}{f_{2\ell+1}(s)}
=\frac{f_{2\ell+1}(s)}{f_{2m+1}(s)}\biggl[\frac{f_{2m+1}(s)}{f_{2\ell+1}(s)}\biggr]'.
\end{equation*}
From Corollary~\ref{B(2m-1)(2n-1)-cor}, it follows that the ratio
\begin{equation*}
\frac{f_{2m+1}(s)}{f_{2\ell+1}(s)}=(-1)^{\ell-m}\frac{(2\ell+1)!}{(2m+1)!}\frac{B_{2m+1}(s)}{B_{2\ell+1}(s)}
\end{equation*}
is positive and decreases in $s\in\bigl(0,\frac{1}{2}\bigr)$, and is positive and increases in $s\in\bigl(\frac{1}{2},1\bigr)$. Accordingly, we obtain
\begin{equation}\label{f-ratio-Bernoul-atio}
\biggl[\frac{f_{2m+1}(s)}{f_{2\ell+1}(s)}\biggr]'
\begin{dcases}
<0, & s\in\biggl(0,\frac{1}{2}\biggr)\\
>0, & s\in\biggl(\frac{1}{2},1\biggr)
\end{dcases}
\end{equation}
for $m,\ell\in\mathbb{N}_{0}$ with $m<\ell$.
Consequently, for given $s\in\bigl(0,\frac12\bigr)$ (or for given $s\in\bigl(\frac{1}{2},1\bigr)$, respectively), the sequence $\frac{(2\ell+1)B_{2\ell}(s)}{B_{2\ell+1}(s)}$ increases (or decreases, respectively) in $\ell\in\mathbb{N}_0$.
\par
By virtue of the asymptotic approximations in~\eqref{p594Entry24.11.5NIST-HB-2010}, we obtain
\begin{equation*}
\frac{(2\ell+1)B_{2\ell}(s)}{B_{2\ell+1}(s)}
\sim \frac{(2\ell+1)\cos(2\pi s)}{(-1)^{\lfloor \ell\rfloor-1}\frac{(2\pi)^{2\ell}}{2(2\ell)!}}
\frac{(-1)^{\lfloor \ell+1/2\rfloor-1}\frac{(2\pi)^{2\ell+1}}{2(2\ell+1)!}} {\sin(2\pi s)}
=2\pi\cot(2\pi s)
\end{equation*}
as $\ell\to\infty$ for $s\in(0,1)\setminus\bigl\{\frac{1}{2}\bigr\}$.
The required proof is complete.
\end{proof}

\begin{cor}
For $\ell\in\mathbb{N}_0$, the absolute polynomial $|B_{2\ell+1}(s)|$ is logarithmically concave in $s\in\bigl(0,\frac{1}{2}\bigr)$ and in $s\in\bigl(\frac{1}{2},1\bigr)$.
\end{cor}

\begin{proof}
On the interval $\bigl(0,\frac{1}{2}\bigr)$, we have
\begin{equation*}
[\ln |B_{2\ell+1}(s)|]'=\bigl(\ln\bigl[(-1)^{\ell+1}B_{2\ell+1}(s)\bigr]\bigr)'
=\frac{B_{2\ell+1}'(s)}{B_{2\ell+1}(s)}
=\frac{(2\ell+1)B_{2\ell}(s)}{B_{2\ell+1}(s)},
\end{equation*}
where we used the equality~\eqref{Bernou-polyn-deriv}, which, in light of the first paragraph of Theorem~\ref{T5}, decreases in $s\in\bigl(0,\frac{1}{2}\bigr)$ for $\ell\in\mathbb{N}_0$. As a result, the absolute polynomial $|B_{2\ell+1}(s)|$ is logarithmically concave in $s\in\bigl(0,\frac{1}{2}\bigr)$ for $\ell\in\mathbb{N}_0$.
\par
On the interval $\bigl(\frac{1}{2},1\bigr)$, by virtue of the identity~\eqref{Bern-polyn-symm-ID} and~\eqref{Bernou-polyn-deriv} in sequence, we have
\begin{gather*}
[\ln |B_{2\ell+1}(s)|]'=\bigl(\ln\bigl|(-1)^{\ell+1}B_{2\ell+1}(1-s)\bigr|\bigr)'
=\bigl(\ln\bigl[(-1)^{\ell+1}B_{2\ell+1}(1-s)\bigr]\bigr)'\\
=\frac{(-1)^{\ell+2}B_{2\ell+1}'(1-s)}{(-1)^{\ell+1}B_{2\ell+1}(1-s)}
=\frac{(-1)^{\ell+2}(2\ell+1)B_{2\ell}(1-s)}{(-1)^{\ell+1}B_{2\ell+1}(1-s)}
=-\frac{(2\ell+1)B_{2\ell}(1-s)}{B_{2\ell+1}(1-s)}
\end{gather*}
in $s\in\bigl(\frac{1}{2},1\bigr)$ for $\ell\in\mathbb{N}_0$, which, in light of the first paragraph of Theorem~\ref{T5}, decreases in $s\in\bigl(\frac{1}{2},1\bigr)$ for $\ell\in\mathbb{N}_0$. Consequently, the absolute polynomial $|B_{2\ell+1}(s)|$ is also logarithmically concave in $s\in\bigl(\frac{1}{2},1\bigr)$ for $\ell\in\mathbb{N}_0$. The required proof is complete.
\end{proof}

\begin{thm}\label{T3}
For $\ell\in\mathbb{N}$, let $r_{2\ell}$ be the unique zero of the Bernoulli polynomial $B_{2\ell}(s)$ on $\bigl(0,\frac12\bigr)$.
For $m,\ell\in\mathbb{N}_{0}$ of $m<\ell$, the ratio $(-1)^{\ell-m}\frac{B_{2m}(s)}{B_{2\ell}(s)}$ decreases in $s\in(0,r_{2\ell})$ and in $s\in\bigl(r_{2\ell},\frac12\bigr)$, while it increases in $s\in\bigl(\frac12,1-r_{2\ell}\bigr)$ and in $s\in(1-r_{2\ell},1)$.
\end{thm}

\begin{proof}
In view of the notations in~\eqref{f2n=f2n+1} and with aid of the relations in~\eqref{df2n-2n+1}, we arrive at
\begin{equation*}
H_{f_{2m},f_{2\ell}}(s)=\frac{f_{2m-1}(s)}{f_{2\ell-1}(s)}f_{2\ell}(s)-f_{2m}(s),
\end{equation*}
where $H_{f,g}(s)$ is defined by~\eqref{H(fg)(s)-dfn}.
\par
Using the second relations in~\eqref{df2n-2n+1} and~\eqref{f2n=f2n+1} in sequence and employing~\eqref{Bern-polyn-half-posit}, we obtain
\begin{equation*}
f_{2\ell}'(s)=(-1)^{\ell+1}\frac{B_{2\ell-1}(s)}{(2\ell-1)!}<0
\end{equation*}
for $\ell\in\mathbb{N}$ and $0<s<\frac{1}{2}$. Hence, the function $f_{2\ell}(s)$ decreases in $s\in\bigl(0,\frac12\bigr)$ and, by virtue of Lemma~\ref{Ostrowski-1960-Murch-lem}, $f_{2\ell}(r_{2\ell})=0$ for $\ell\in\mathbb{N}$.
Making use of this and the conclusions in~\eqref{f-ratio-Bernoul-atio}, we acquire
\begin{equation*}
[H_{f_{2m},f_{2\ell}}(s)]'=\biggl[\frac{f_{2m-1}(s)}{f_{2\ell-1}(s)}\biggr]'f_{2\ell}(s)
\begin{dcases}
<0, & s\in(0,r_{2\ell}), \\
>0, & s\in\biggl(r_{2\ell},\frac12\biggr)
\end{dcases}
\end{equation*}%
and%
\begin{equation*}
H_{f_{2m},f_{2\ell}}(r_{2\ell}) =\frac{f_{2m-1}(r_{2\ell})}{f_{2\ell-1}(r_{2\ell})}f_{2\ell}(r_{2\ell}) -f_{2m}(r_{2\ell}) =-f_{2m}(r_{2\ell})
\end{equation*}
for $m,\ell\in\mathbb{N}$ such that $m<\ell$.
Since $r_{2\ell}>r_{2m}$ for $m,\ell\in\mathbb{N}$ such that $m<\ell$, then $f_{2m}(r_{2\ell})<f_{2m}(r_{2m})=0$ for $m,\ell\in\mathbb{N}$ such that $m<\ell$. Thus, we deduce
\begin{equation}\label{f2m-2n-r-2n}
H_{f_{2m},f_{2\ell}}(s)\ge H_{f_{2m},f_{2\ell}}(r_{2\ell})=-f_{2m}(r_{2\ell})>0
\end{equation}
for $0<s<\frac{1}{2}$ and $m,\ell\in\mathbb{N}$ such that $m<\ell$. Consequently, by~\eqref{H(fg)(s)fg-dfn}, we conclude that%
\begin{equation*}
\biggl[\frac{f_{2m}(s)}{f_{2\ell}(s)}\biggr]'=-\frac{f_{2\ell-1}(s)}{f_{2\ell}^2(s)}H_{f_{2m},f_{2\ell}}(s)<0
\end{equation*}%
for $s\in(0,r_{2\ell})\cup\bigl(r_{2\ell},\frac12\bigr)$ and $m,\ell\in\mathbb{N}$ such that $m<\ell$, where we utilized the fact that the function $f_{2\ell+1}(s)$ is positive for $\ell\in\mathbb{N}_0$ and $0<s<\frac{1}{2}$, which was derived in the proof of Theorem~\ref{T5}. This means that the ratio
\begin{equation*}
\frac{f_{2m}(s)}{f_{2\ell}(s)}=(-1)^{\ell-m}\frac{B_{2m}(s)}{B_{2\ell}(s)} \frac{(2\ell)!}{(2m)!}
\end{equation*}
for $m,\ell\in\mathbb{N}$ such that $m<\ell$ decreases in $s\in(0,r_{2\ell}) $ and $s\in\bigl(r_{2\ell},\frac12\bigr)$.
\par
When $m=0$ and $\ell\in\mathbb{N}$, it is clear that the ratio
\begin{equation*}
(-1)^{\ell-m}\frac{B_{2m}(s)}{B_{2\ell}(s)}
=\frac{1}{(-1)^\ell B_{2\ell}(s)}
=-\frac{1}{(2\ell)!}\frac{1}{f_{2\ell}(s)}
\end{equation*}
decreases in $s\in(0,r_{2\ell}) $ and in $s\in\bigl(r_{2\ell},\frac12\bigr)$.
\par
The rest can be proved by considering the symmetry expressed in~\eqref{Bern-polyn-symm-ID}.
\end{proof}

\begin{thm}\label{T6}
For $\ell\in\mathbb{N}$, the ratio $\frac{B_{2\ell}(s)}{B_{2\ell-1}(s)}$ increases in $s\in\bigl(0,\frac12\bigr)$ and in $s\in\bigl(\frac{1}{2},1\bigr)$.
\par
For fixed $s\in\bigl(0,\frac12\bigr)$ \textup{(}or for fixed $s\in\bigl(\frac{1}{2},1\bigr)$, respectively\textup{)}, the sequence $\frac{B_{2\ell}(s)}{\ell B_{2\ell-1}(s)}$ decreases \textup{(}or increases, respectively\textup{)} in $\ell\in\mathbb{N}$, with the limit
\begin{equation}\label{limit-Ber2n=2n-1-cot}
\lim_{\ell\to\infty}\frac{B_{2\ell}(s)}{\ell B_{2\ell-1}(s)}
=-\frac{\cot(2\pi s)}{\pi}.
\end{equation}
\end{thm}

\begin{proof}
In view of the notations in~\eqref{f2n=f2n+1} and with aid of the relations in~\eqref{df2n-2n+1}, directly differentiating yields
\begin{equation*}
\biggl[\frac{f_{2\ell}(s)}{f_{2\ell-1}(s)}\biggr]'=\frac{-f_{2\ell-1}^2(s)-f_{2\ell}(s)f_{2\ell-2}(s)}{f_{2\ell-1}^2(s)}
\triangleq \frac{g_\ell(s)}{f_{2\ell-1}^2(s)}
\end{equation*}
and
\begin{multline*}
g_\ell'(s)=-f_{2\ell-1}(s)f_{2\ell-2}(s) +f_{2\ell}(s)f_{2\ell-3}(s)\\
=-f_{2\ell}^2(s)\frac{f_{2\ell}(s)f_{2\ell-2}'(s)-f_{2\ell}'(s)f_{2\ell-2}(s)}{f_{2\ell}^2(s)}
=-f_{2\ell}^2(s)\biggl[\frac{f_{2\ell-2}(s)}{f_{2\ell}(s)}\biggr]'
\end{multline*}
for $s\in\bigl(0,\frac12\bigr)\setminus\{r_{2\ell}\}$ and $\ell\in\mathbb{N}$, where $r_{2\ell}$ is the unique zero of $B_{2\ell}(s)$ in $s\in\bigl(0,\frac12\bigr)$.
\par
By virtue of Theorem~\ref{T3}, we see that the ratio $\frac{f_{2\ell-2}(s)}{f_{2\ell}(s)}$ for $\ell\in\mathbb{N}$ decreases on $(0,r_{2\ell}) $ and $\bigl(r_{2\ell},\frac12\bigr)$. This implies that $g_\ell'(s)>0$ for $s\in\bigl(0,\frac12\bigr)\setminus\{r_{2\ell}\}$. Due to the continuity of $g_\ell'(s)$ on $\bigl(0,\frac12\bigr)$, we acquire $g_\ell'(s)>0$ for $s\in\bigl(0,\frac12\bigr)$, and then
\begin{equation*}
g_\ell(s)<g_\ell\biggl(\frac12\biggr) =-f_{2\ell}\biggl(\frac12\biggr) f_{2\ell-2}\biggl(\frac12\biggr) <0
\end{equation*}%
for $s\in\bigl(0,\frac12\bigr)$, which indicates that the function
\begin{equation*}
\frac{f_{2\ell}(s)}{f_{2\ell-1}(s)}=-\frac{B_{2\ell}(s)}{2\ell B_{2\ell-1}(s)}
\end{equation*}
decreases in $s\in\bigl(0,\frac12\bigr)$ for $\ell\in\mathbb{N}$. Equivalently, the ratio $\frac{B_{2\ell}(s)}{B_{2\ell-1}(s)}$ for $\ell\in\mathbb{N}$ increases in $s\in\bigl(0,\frac12\bigr)$.
\par
Theorem~\ref{T3} implies that, for $\ell>m\in\mathbb{N}$,
\begin{equation*}
\frac{f_{2\ell}^2(s)}{f_{2\ell-1}(s)f_{2m-1}(s)}\biggl[\frac{f_{2m}(s)}{f_{2\ell}(s)}\biggr]'
=\frac{f_{2m}(s)}{f_{2m-1}(s)}-\frac{f_{2\ell}(s)}{f_{2\ell-1}(s)}<0
\end{equation*}%
in $s\in\bigl(0,\frac12\bigr)\setminus\{r_{2\ell}\}$, which means that the sequence $\frac{f_{2\ell}(s)}{f_{2\ell-1}(s)}$ increases in $\ell\in\mathbb{N}$ for fixed $s\in\bigl(0,\frac12\bigr)$.
\par
The limit~\eqref{limit-Ber2n=2n-1-cot} follows from~\eqref{p594Entry24.11.5NIST-HB-2010}.
\par
The rest follows from applying the identity~\eqref{Bern-polyn-symm-ID}. The proof is complete.
\end{proof}

\section{Inequalities of Bernoulli polynomials and their ratios}\label{inequalities-sec-bernoulli}

In this section, using the monotonicity results established in Sections~\ref{Qi-probl-3rd-proof-sec} and~\ref{monot-two-bern-polyn-sec}, we derive several new inequalities and recover a number of known inequalities for the Bernoulli polynomials $B_\ell(s)$, the Bernoulli numbers $B_{2\ell}$, and their ratios such as $\frac{B_{2\ell+2}}{B_{2\ell}}$.

\begin{thm}\label{C-B2n+1-b1}
For $\ell\ge2$, the function $\frac{|B_{2\ell+1}(s)|}{s(1/2-s)(1-s)}$ increases in $s\in\bigl(0,\frac12\bigr)$ and decreases in $s\in\bigl(\frac12,1\bigr)$. Consequently, for $\ell\ge2$, the double inequality
\begin{equation}\label{B2n+1-b1}
2(2\ell+1)|B_{2\ell}|
<\frac{|B_{2\ell+1}(s)|}{s\bigl(\frac12-s\bigr)(1-s)}
<4\biggl(1-\frac{1}{2^{2\ell-1}}\biggr)(2\ell+1)|B_{2\ell}|
\end{equation}
validates in $s\in\bigl(0,\frac12\bigr)$ and reverses in $s\in\bigl(\frac{1}{2},1\bigr)$.
\par
For $\ell\ge2$, the inequality
\begin{equation}\label{B2n+1<I1a}
|B_{2\ell+1}(s)|<\frac{\sqrt{3}\,}{9} \biggl(1-\frac{1}{2^{2\ell-1}}\biggr)(2\ell+1)|B_{2\ell}|
\end{equation}%
validates for $s\in\bigl(0,\frac12\bigr)$.
\end{thm}

\begin{proof}
It is straightforward that
\begin{equation*}
B_{3}(s)=s\biggl(\frac{1}{2}-s\biggr)(1-s).
\end{equation*}
From Corollary~\ref{B(2m-1)(2n-1)-cor}, we see that the ratio
\begin{equation*}
(-1)^{\ell-2}\frac{B_3(s)}{B_{2\ell-1}(s)}
=(-1)^\ell\frac{s\bigl(\frac{1}{2}-s\bigr)(1-s)}{B_{2\ell-1}(s)}, \quad \ell>2
\end{equation*}
is positive and decreases in $s\in\bigl(0,\frac{1}{2}\bigr)$, meanwhile it is positive and increases in $s\in\bigl(\frac{1}{2},1\bigr)$. The monotonicity in Theorem~\ref{C-B2n+1-b1} is thus proved.
\par
By virtue of the L'H\^opital rule and the formulas~\eqref{Bernou-polyn-deriv} and~\eqref{Bernoulli-polyn-numb}, we acquire
\begin{equation*}
\lim_{s\to0^+}\frac{B_{2\ell+1}(s)}{s\bigl(\frac{1}{2}-s\bigr)(1-s)}
=2\lim_{s\to0^+}(2\ell+1)B_{2\ell}(s)
=2(2\ell+1)B_{2\ell}
\end{equation*}
and
\begin{multline*}
\lim_{s\to\frac12}\frac{B_{2\ell+1}(s)}{s\bigl(\frac{1}{2}-s\bigr)(1-s)}
=4\lim_{s\to\frac12}(2\ell+1)B_{2\ell}(s)\\
=-4(2\ell+1)B_{2\ell}\biggl(\frac{1}{2}\biggr)
=4(2\ell+1)\biggl(1-\frac{1}{2^{2\ell-1}}\biggr)B_{2\ell}
\end{multline*}
for $\ell\ge2$. Combining these two limits with the above monotonicity yields the double inequality~\eqref{B2n+1-b1}.
\par
Since the polynomial $B_3(s)$ reaches the maximum value $\frac{\sqrt{3}\,}{36}$ at $s=\frac{3-\sqrt{3}\,}{6}$ on $\bigl(0,\frac12\bigr)$, the right-hand side of the double inequality~\eqref{B2n+1-b1} becomes
\begin{multline*}
|B_{2\ell+1}(s)|
<4\biggl(1-\frac{1}{2^{2\ell-1}}\biggr)(2\ell+1)|B_{2\ell}|s\biggl(\frac12-s\biggr)(1-s)\\
\le4\biggl(1-\frac{1}{2^{2\ell-1}}\biggr)(2\ell+1)|B_{2\ell}|\frac{\sqrt{3}\,}{36}
=\frac{\sqrt{3}\,}{9} \biggl(1-\frac{1}{2^{2\ell-1}}\biggr) (2\ell+1)|B_{2\ell}|
\end{multline*}
for $\ell\ge2$ and $s\in\bigl(0,\frac12\bigr)$. The inequality~\eqref{B2n+1<I1a} is thus proved.
\end{proof}

\begin{thm}\label{C-B2n+1-b2}
For $\ell\in\mathbb{N}_{0}$, the double inequality
\begin{equation}\label{B2n+1-b2}
\biggl(1-\frac1{2^{2\ell-1}}\biggr)\frac{2\ell+1}{2\pi}|B_{2\ell}|\sin(2\pi s)
<|B_{2\ell+1}(s)|
<\frac{2\ell+1}{2\pi}|B_{2\ell}|\sin(2\pi s)
\end{equation}
validates in $s\in\bigl(0,\frac12\bigr)$ and reverses in $s\in\bigl(\frac{1}{2},1\bigr)$.
\end{thm}

\begin{proof}
By Corollary~\ref{B(2m-1)(2n-1)-cor}, we see that the double inequality
\begin{equation}\label{doub-ineq-bern}
(-1)^{\ell-m}\lim_{s\to(\frac{1}{2})^-}\frac{B_{2m-1}(s)}{B_{2\ell-1}(s)}
<(-1)^{\ell-m}\frac{B_{2m-1}(s)}{B_{2\ell-1}(s)}
<(-1)^{\ell-m}\lim_{s\to0^+}\frac{B_{2m-1}(s)}{B_{2\ell-1}(s)}
\end{equation}
holds on $\bigl(0,\frac{1}{2}\bigr)$ for $m,\ell\in\mathbb{N}$ such that $m<\ell$. Since, by the L'H\^opital rule,
\begin{equation*}
\lim_{s\to\frac{1}{2}}\frac{B_{2m-1}(s)}{B_{2\ell-1}(s)}
=\lim_{s\to\frac{1}{2}}\frac{(2m-1)B_{2m-2}(s)}{(2\ell-1)B_{2\ell-2}(s)}
=\frac{(2m-1)B_{2m-2}\bigl(\frac{1}{2}\bigr)}{(2\ell-1)B_{2\ell-2}\bigl(\frac{1}{2}\bigr)}
\end{equation*}
and
\begin{equation*}
\lim_{s\to0^+}\frac{B_{2m-1}(s)}{B_{2\ell-1}(s)}
=\lim_{s\to0^+}\frac{(2m-1)B_{2m-2}(s)}{(2\ell-1)B_{2\ell-2}(s)}
=\frac{(2m-1)B_{2m-2}}{(2\ell-1)B_{2\ell-2}}
\end{equation*}
for $m,\ell\in\mathbb{N}$ such that $m<\ell$, the double inequality~\eqref{doub-ineq-bern} is equivalent to
\begin{multline}\label{Bernoulli-double-2m-1-ineq}
(-1)^{m}(2m-1)B_{2m-2}\biggl(\frac{1}{2}\biggr)\frac{B_{2\ell-1}(s)}{(2\ell-1)B_{2\ell-2}\bigl(\frac{1}{2}\bigr)}
<|B_{2m-1}(s)|\\
<(-1)^{m}(2m-1)B_{2m-2}\frac{B_{2\ell-1}(s)}{(2\ell-1)B_{2\ell-2}}
\end{multline}
on $\bigl(0,\frac{1}{2}\bigr)$ for $m,\ell\in\mathbb{N}$ such that $m<\ell$.
\par
When taking $\ell\to\infty$, by asymptotic approximations in~\eqref{p594Entry24.11.5NIST-HB-2010}, we obtain
\begin{equation*}
\frac{B_{2\ell-1}(s)}{(2\ell-1)B_{2\ell-2}\bigl(\frac{1}{2}\bigr)}\sim \frac{(-1)^{\lfloor \ell-1/2\rfloor-1}\frac{2(2\ell-1)!}{(2\pi)^{2\ell-1}}\sin(2\pi s)} {(2\ell-1)(-1)^{\lfloor \ell-1\rfloor-1}\frac{2(2\ell-2)!}{(2\pi)^{2\ell-2}}\cos\pi}
=-\frac{\sin(2\pi s)}{2\pi}
\end{equation*}
and
\begin{equation*}
\frac{B_{2\ell-1}(s)}{(2\ell-1)B_{2\ell-2}}\sim \frac{(-1)^{\lfloor \ell-1/2\rfloor-1}\frac{2(2\ell-1)!}{(2\pi)^{2\ell-1}}\sin(2\pi s)} {(2\ell-1)(-1)^{\lfloor \ell-1\rfloor-1}\frac{2(2\ell-2)!}{(2\pi)^{2\ell-2}}\cos 0}
=\frac{\sin(2\pi s)}{2\pi}
\end{equation*}
on $\bigl(0,\frac{1}{2}\bigr)$. Applying these two results to the double inequality~\eqref{Bernoulli-double-2m-1-ineq} and utilizing~\eqref{Bernoulli-polyn-numb} yield
\begin{multline}\label{Bern-doub-2m-1-ineq}
(-1)^{m+1}(2m-1)B_{2m-2}\biggl(\frac{1}{2}\biggr)\frac{\sin(2\pi s)}{2\pi}\\
=(-1)^{m}(2m-1)\biggl(1-\frac{1}{2^{2m-3}}\biggr)B_{2m-2}\frac{\sin(2\pi s)}{2\pi}\\
<|B_{2m-1}(s)|
<(-1)^{m}(2m-1)B_{2m-2}\frac{\sin(2\pi s)}{2\pi}
\end{multline}
for $m\in\mathbb{N}$. Replacing $m$ by $\ell+1$ in~\eqref{Bern-doub-2m-1-ineq} leads to the double inequality~\eqref{B2n+1-b2}.
\par
By the same argument as above and in view of the identity~\eqref{Bern-polyn-symm-ID}, we can derive the reversed version of~\eqref{B2n+1-b2} on $\bigl(\frac{1}{2},1\bigr)$ for $\ell\in\mathbb{N}_{0}$. The required proof is thus complete.
\end{proof}

\begin{rem}
Proposition~4.3 in~\cite{Kouba-arxiv-2016} reads that, for every positive integer $\ell$,
\begin{equation}\label{Kouba-arxiv-2016-ineq1}
\sup_{x\in[0,1]}|B_{2\ell+1}(s)| <\frac{2\ell+1}{2\pi}|B_{2\ell}|
\end{equation}
and
\begin{equation}\label{Kouba-arxiv-2016-ineq2}
\biggl|B_{2\ell+1}\biggl(\frac{1}{4}\biggr)\biggr|\ge\biggl(1-\frac{4}{2^{2\ell}}\biggr)\frac{2\ell+1}{2\pi}|B_{2\ell}|.
\end{equation}
\par
From Theorem~\ref{C-B2n+1-b2}, it follows that
\begin{align*}
|B_{2\ell+1}(s)|&\le \max\biggl\{\frac{2\ell+1}{2\pi}|B_{2\ell}|\sin(2\pi s),\biggl(1-\frac1{2^{2\ell-1}}\biggr)\frac{2\ell+1}{2\pi}|B_{2\ell}|\sin(2\pi s)\biggr\}\\
&\le\frac{2\ell+1}{2\pi}|B_{2\ell}|
\end{align*}
for $s\in(0,1)$ and $\ell\in\mathbb{N}_{0}$. This means that Theorem~\ref{C-B2n+1-b2} improves the inequality~\eqref{Kouba-arxiv-2016-ineq1}.
\par
Taking $s=\frac14$ in~\eqref{B2n+1-b2} gives
\begin{equation*}
\biggl(1-\frac{2}{2^{2\ell}}\biggr)\frac{2\ell+1}{2\pi}|B_{2\ell}|
<\biggl|B_{2\ell+1}\biggl(\frac{1}{4}\biggr)\biggr|
<\frac{2\ell+1}{2\pi}|B_{2\ell}|,\quad \ell\in\mathbb{N}_0.
\end{equation*}
This means that Theorem~\ref{C-B2n+1-b2} improves the inequality~\eqref{Kouba-arxiv-2016-ineq2}.
\par
In conclusion, Theorem~\ref{C-B2n+1-b2} in this paper is better than~\cite[Proposition~4.3]{Kouba-arxiv-2016}.
\end{rem}

\begin{thm}
For $\ell\in\mathbb{N}_{0}$, we have
\begin{equation}\label{B2n-b3a}
(-1)^{\ell+1}B_{2\ell}(s)<|B_{2\ell}|\cos(2\pi s), \quad s\in\biggl(0,\frac14\biggr)
\end{equation}
and
\begin{equation}\label{B2n-b3b}
(-1)^{\ell+1}B_{2\ell}(s)<\biggl(1-\frac1{2^{2\ell-1}}\biggr)|B_{2\ell}|\cos(2\pi s), \quad s\in \biggl(\frac14,\frac12\biggr).
\end{equation}
\end{thm}

\begin{proof}
From Theorem~\ref{T3}, it follows that
\begin{equation*}
(-1)^{\ell-m}\frac{B_{2m}(s)}{B_{2\ell}(s)}<(-1)^{\ell-m}\frac{B_{2m}}{B_{2\ell}}, \quad s\in(0,r_{2\ell})
\end{equation*}
and
\begin{equation*}
(-1)^{\ell-m}\frac{B_{2m}(s)}{B_{2\ell}(s)}>(-1)^{\ell-m}\frac{B_{2m}\bigl(\frac12\bigr)}{B_{2\ell}\bigl(\frac12\bigr)}, \quad s\in\biggl(r_{2\ell},\frac12\biggr)
\end{equation*}
for $m,\ell\in\mathbb{N}_{0}$ such that $m<\ell$.
They can be reformulated as
\begin{equation*}
(-1)^{m+1}B_{2m}(s)<\frac{B_{2\ell}(s)}{B_{2\ell}}|B_{2m}|, \quad s\in(0,r_{2\ell})
\end{equation*}
and
\begin{equation*}
(-1)^{m+1}B_{2m}(s)<\frac{B_{2\ell}(s)}{B_{2\ell}\bigl(\frac12\bigr)}(-1)^{m+1}B_{2m}\biggl(\frac12\biggr),\quad s\in\biggl(r_{2\ell},\frac12\biggr)
\end{equation*}
for $m,\ell\in\mathbb{N}_{0}$ such that $m<\ell$, where we used the fact that the function $f_{2\ell}(s)$ decreases in $s\in\bigl(0,\frac12\bigr)$ and $f_{2\ell}(r_{2\ell})=0$ for $\ell\in\mathbb{N}$.
Making use of asymptotic approximations in~\eqref{p594Entry24.11.5NIST-HB-2010} gives, as $\ell\to\infty $,
\begin{equation}\label{sim-even-even1}
\frac{B_{2\ell}(s)}{B_{2\ell}}\sim \frac{(-1)^{\lfloor \ell\rfloor-1}\frac{2(2\ell)!}{(2\pi)^{2\ell}}\cos(2\pi s)} {(-1)^{\lfloor \ell\rfloor-1}\frac{2(2\ell)!}{(2\pi)^{2\ell}}\cos0}=\cos(2\pi s)
\end{equation}
and
\begin{equation}\label{sim-even-even2}
\frac{B_{2\ell}(s)}{B_{2\ell}\bigl(\frac12\bigr)}\sim \frac{(-1)^{\lfloor \ell\rfloor-1}\frac{2(2\ell)!}{(2\pi)^{2\ell}}\cos(2\pi s)} {(-1)^{\lfloor \ell\rfloor-1}\frac{2(2\ell)!}{(2\pi)^{2\ell}}\cos\pi}=-\cos(2\pi s).
\end{equation}
Then, due to the fact $r_{2\ell}\to\frac14$ as $\ell\to\infty$,
\begin{equation*}
(-1)^{m+1}B_{2m}(s)<|B_{2m}|\cos(2\pi s),\quad s\in\biggl(0,\frac14\biggr)
\end{equation*}
and
\begin{equation*}
(-1)^{m+1}B_{2m}(s)<\biggl(1-\frac1{2^{2m-1}}\biggr)|B_{2m}|\cos(2\pi s), \quad s\in\biggl(\frac14,\frac12\biggr).
\end{equation*}
The required proof is complete.
\end{proof}

\begin{thm}\label{C-B2n-b2}
For $s\in\bigl(0,\frac12\bigr)\cup\bigl(\frac12,1\bigr)$, we have
\begin{equation}\label{B2n-b4a}
\ell(2\ell-1)|B_{2\ell-2}|<(-1)^\ell\frac{B_{2\ell}(s)-B_{2\ell}}{s^2(1-s)^2}
<32\biggl(1-\frac{1}{2^{2\ell}}\biggr) |B_{2\ell}|, \quad \ell\ge3
\end{equation}
and
\begin{multline}\label{B2n-b4b}
8\biggl(1-\frac{1}{2^{2\ell}}\biggr) |B_{2\ell}|<(-1)^{\ell+1}\frac{B_{2\ell}(s)-B_{2\ell}\bigl(\frac12\bigr)}{\bigl(s-\frac12\bigr)^2}\\
<\ell(2\ell-1) \biggl(1-\frac1{2^{2\ell-3}}\biggr) |B_{2\ell-2}|, \quad \ell\ge2.
\end{multline}
\end{thm}

\begin{proof}
It is straightforward that
\begin{equation*}
B_{4}(s)-B_{4}=s^2(s-1)^2
\quad\text{and}\quad
B_{2}(s) -B_{2}\biggl(\frac12\biggr)=\biggl(s-\frac12\biggr)^2.
\end{equation*}
By Corollary~\ref{T4}, the functions
\begin{equation*}
\frac{(-1)^\ell[B_{2\ell}(s)-B_{2\ell}]}{s^2(s-1)^2}, \quad \ell\ge 3
\end{equation*}
and
\begin{equation*}
\frac{(-1)^{\ell+1}\bigl[B_{2\ell}(s)-B_{2\ell}\bigl(\frac12\bigr)\bigr]}{\bigl(s-\frac12\bigr)^2},\quad \ell\ge 2
\end{equation*}
increase in $s\in\bigl(0,\frac12\bigr)$ and decrease in $s\in\bigl(\frac12,1\bigr)$. With help of the L'H\^opital rule and by virtue of the formulas~\eqref{Bernou-polyn-deriv} and~\eqref{Bernoulli-polyn-numb}, we arrive at the limits
\begin{align*}
\lim_{s\to\frac12}\frac{(-1)^\ell[B_{2\ell}(s) -B_{2\ell}]}{s^2(s-1)^2} &=32\biggl(1-\frac{1}{2^{2\ell}}\biggr)|B_{2\ell}|,\\
\lim_{s\to0^+}\frac{(-1)^{\ell+1}[B_{2\ell}(s) -B_{2\ell}\bigl(\frac12\bigr)\bigr]} {\bigl(s-\frac12\bigr)^2}
&=8\biggl(1-\frac1{2^{2\ell}}\biggr)|B_{2\ell}|\\
&=\lim_{s\to1^-}\frac{(-1)^{\ell+1}[B_{2\ell}(s) -B_{2\ell}\bigl(\frac12\bigr)\bigr]} {\bigl(s-\frac12\bigr)^2},\\
\lim_{s\to\frac12}\frac{(-1)^{\ell+1}\bigl[B_{2\ell}(s) -B_{2\ell}\bigl(\frac12\bigr)\bigr]}{\bigl(s-\frac12\bigr)^2}
&=\ell(2\ell-1) \biggl(1-\frac1{2^{2\ell-3}}\biggr)|B_{2\ell-2}|
\end{align*}
for $\ell\in\mathbb{N}$ and the limits
\begin{equation*}
\lim_{s\to0^+}\frac{(-1)^\ell[B_{2\ell}(s)-B_{2\ell}]}{s^2(s-1)^2}
=\ell(2\ell-1)|B_{2\ell-2}|
=\lim_{s\to1^-}\frac{(-1)^\ell[B_{2\ell}(s)-B_{2\ell}]}{s^2(s-1)^2}
\end{equation*}
for $\ell\ge2$, where we excluded the case $\lim_{s\to0^+}B_{1}(s)=B_1=-\frac{1}{2}\ne0$.
Hence, the double inequalities~\eqref{B2n-b4a} and~\eqref{B2n-b4b} are proved. The proof of Theorem~\ref{C-B2n-b2} is thus complete.
\end{proof}

\begin{rem}
Since $s^2(1-s)^2<\frac1{16}$ for $s\in\bigl(0,\frac12\bigr)\cup\bigl(\frac12,1\bigr)$, the right-hand side of the inequality~\eqref{B2n-b4a} can be weaken to%
\begin{equation}\label{Bernou-diff-ineq>3}
(-1)^\ell[B_{2\ell}(s)-B_{2\ell}]<2\biggl(1-\frac1{2^{2\ell}}\biggr) |B_{2\ell}|,\quad \ell\ge3
\end{equation}
for $s\in\bigl(0,\frac12\bigr)\cup\bigl(\frac12,1\bigr)$. Employing the fact that the function $f_{2\ell}(s)=(-1)^{\ell+1}\frac{B_{2\ell}(s)}{(2\ell)!}$ decreases in $s\in\bigl(0,\frac12\bigr)$ for $\ell\in\mathbb{N}$, which was proved in the proof of Theorem~\ref{T3}, considering the equality~\eqref{Bern-polyn-symm-ID}, and concretely computing the cases $\ell=1,2$ and the cases $s=0,\frac{1}{2},1$, we can reformulate the inequality~\eqref{Bernou-diff-ineq>3} as
\begin{equation}\label{Bernou-diff-ineq-rewr}
|B_{2\ell}(s)-B_{2\ell}|\le\biggl(2-\frac{1}{2^{2\ell-1}}\biggr)|B_{2\ell}|,\quad 1\ge s\ge0, \quad \ell\in\mathbb{N}.
\end{equation}
We emphasize that the inequality~\eqref{Bernou-diff-ineq-rewr} appeared in~\cite[p.~593, Entry~24.9.2]{NIST-HB-2010}.
\end{rem}

\begin{rem}
We can rewrite the right-hand side inequality of~\eqref{B2n-b4a} as
\begin{equation}\label{Lower-One-Yang}
(-1)^{\ell+1}B_{2\ell}(s)
>\biggl[1-32\biggl(1-\frac{1}{2^{2\ell}}\biggr)s^2(1-s)^2\biggr]|B_{2\ell}|, \quad \ell\ge3
\end{equation}
for $s\in\bigl(0,\frac12\bigr)\cup\bigl(\frac12,1\bigr)$. Utilizing the identity~\eqref{Bernoulli-polyn-numb}, we can also rewrite the left-hand side of the double inequality~\eqref{B2n-b4b} as
\begin{equation}\label{Lower-Two-Yang}
(-1)^{\ell+1}B_{2\ell}(s)
>\biggl[8\biggl(1-\frac{1}{2^{2\ell}}\biggr) \biggl(s-\frac12\biggr)^2-\biggl(1-\frac{1}{2^{2\ell-1}}\biggr)\biggr]|B_{2\ell}|, \quad \ell\ge2
\end{equation}
for $s\in\bigl(0,\frac12\bigr)\cup\bigl(\frac12,1\bigr)$. The lower bound in~\eqref{Lower-One-Yang} is bigger than the corresponding one in~\eqref{Lower-Two-Yang} for $s\in\bigl(0,\frac12\bigr)\cup\bigl(\frac12,1\bigr)$ and $\ell\in\mathbb{N}$.
\end{rem}

\begin{rem}
On the left-hand side of the double inequality~\eqref{B2n-b4a}, taking $s\to\frac{1}{2}$ leads to
\begin{equation*}
\frac{\ell(2\ell-1)}{16}|B_{2\ell-2}|\le(-1)^\ell\biggl[B_{2\ell}\biggl(\frac{1}{2}\biggr)-B_{2\ell}\biggr], \quad \ell\ge3.
\end{equation*}
On the right-hand side of the double inequality~\eqref{B2n-b4b}, letting $s\to0^+$ results in
\begin{equation*}
(-1)^{\ell+1}\biggl[B_{2\ell}-B_{2\ell}\biggl(\frac12\biggr)\biggr]
\le\frac{\ell(2\ell-1)}{4} \biggl(1-\frac1{2^{2\ell-3}}\biggr) |B_{2\ell-2}|, \quad \ell\ge2.
\end{equation*}
Further utilizing the identity~\eqref{Bernoulli-polyn-numb} yields
\begin{equation}\label{(29)(30)-ineq-ratio-Bern}
\frac{2^{2\ell+2}}{2^{2\ell+2}-1}\frac{(\ell+1)(2\ell+1)}{32}
\le\biggl|\frac{B_{2\ell+2}}{B_{2\ell}}\biggr|
\le \frac{2^{2\ell+2}-8}{2^{2\ell+2}-1} \frac{(\ell+1)(2\ell+1)}{8}
\end{equation}
for $\ell\in\mathbb{N}$. Except the case $\ell=1$ on the left-hand side, the double inequality~\eqref{(29)(30)-ineq-ratio-Bern} is weaker than
\begin{equation}\label{ineq-Bernou-equiv}
\frac{2^{2\ell}-2} {2^{2\ell+1}-1}\frac{(\ell+1)(2\ell+1)}{\pi^2}
<\biggl|\frac{B_{2\ell+2}}{B_{2\ell}}\biggr|
<\frac{2^{2\ell+1}-2}{2^{2\ell+2}-1}\frac{(\ell+1)(2\ell+1)}{\pi^2}
\end{equation}
for $\ell\in\mathbb{N}$, which was first established in~\cite[Theorem~1.1]{CAM-D-18-00067.tex}.
\par
We note that the bounds in the double inequality~\eqref{(29)(30)-ineq-ratio-Bern} are rational, whereas the lower and upper bounds in~\eqref{ineq-Bernou-equiv} are irrational.
\par
Moreover, the double inequality~\eqref{ineq-Bernou-equiv} has been elegantly generalized in~\cite{Filomat1193.tex, Rem-Bern-Ratio-Ineq.tex, Yang-Tian-JCAM-2020, Zhu-RACSAM-2020}, concisely reviewed in~\cite{MIA-9509.tex, RCSM-D-21-00302.tex}, and independently cited in~\cite{Qi-collected-12, Qi-collected-4, Qi-collected-23, Qi-collected-13, Qi-collected-34, Bagul-Du-2025, Qi-collected-24, Qi-collected-35, Qi-collected-14, Qi-collected-15, Qi-collected-16, Qi-collected-2, Qi-collected-36, Qi-collected-25, Qi-collected-47, Qi-collected-48, FerrariF-mcom-4062, Qi-collected-3, ESAIM-Ferrari, arXiv.2411.00650, Qi-collected-1, Qi-collected-5, Qi-collected-26, Qi-collected-17, Qi-collected-49, Gulf-Nantomah, Qi-collected-6, Qi-collected-7, Qi-collected-27, arXiv:2508.17938IMRN, arXiv:2508.17938, Qi-collected-28, Qi-collected-37, OpenJMA2026, RGMIA-2024, Xu-Open2026, Qi-collected-38, Yang-JMAA-Corrigendum, arXiv:2507.10954, Qi-collected-8, Qi-collected-9, Zhao-AADM-25, Qi-collected-10, Qi-collected-29, Qi-collected-40, Qi-collected-30, Qi-collected-41, Qi-collected-18, Qi-collected-31, Qi-collected-43, Qi-collected-19, Qi-collected-20, Qi-collected-21, Qi-collected-50, Qi-collected-44, Qi-collected-32, Qi-collected-33, Qi-collected-11, Qi-collected-51, Qi-collected-45, Qi-collected-22}, for example. These independent citations highlight the importance and catalytic role of the double inequality~\eqref{ineq-Bernou-equiv} in mathematics, physics, and other scientific disciplines.
\end{rem}

\begin{thm}
For $\ell\in\mathbb{N}$, the double inequality
\begin{multline}\label{B2n-b5a}
\frac{\ell(2\ell-1)} {2\pi^2}\biggl(1-\frac1{2^{2\ell-3}}\biggr)|B_{2\ell-2}|[1+\cos(2\pi s)] -\biggl(1-\frac1{2^{2\ell-1}}\biggr)|B_{2\ell}|\\
<(-1)^{\ell+1}B_{2\ell}(s)
<\frac{1+\bigl(2^{2\ell}-1\bigr)\cos(2\pi s)}{2^{2\ell}}|B_{2\ell}|
\end{multline}
validates for $s\in\bigl(0,\frac12\bigr)\cup\bigl(\frac{1}{2},1\bigr)$. For $\ell\ge2$, the inequality
\begin{equation}\label{B2n-b5b}
(-1)^{\ell+1}B_{2\ell}(s)>|B_{2\ell}|-\frac{\ell(2\ell-1)}{2\pi^2}|B_{2\ell-2}|[1-\cos(2\pi s)]
\end{equation}
validates for $s\in(0,1)$. When $\ell=1$, the inequality~\eqref{B2n-b5b} reverses for $s\in(0,1)$.
\end{thm}

\begin{proof}
By virtue of the L'H\^opital rule and with aid of the formulas~\eqref{Bernou-polyn-deriv} and~\eqref{Bernoulli-polyn-numb}, a direct computation yields
\begin{equation*}
\lim_{s\to0^+}\frac{B_{2m}(s)-B_{2m}}{B_{2\ell}(s)-B_{2\ell}}
=\frac{m(2m-1)B_{2m-2}}{\ell(2\ell-1) B_{2\ell-2}}
\end{equation*}
and
\begin{equation*}
\lim_{s\to\frac12}\frac{B_{2m}(s) -B_{2m}\bigl(\frac12\bigr)}{B_{2\ell}(s)-B_{2\ell}\bigl(\frac12\bigr)}
=\frac{m(2m-1) \bigl(1-\frac1{2^{2m-3}}\bigr)B_{2m-2}}{\bigl(1-\frac1{2^{2\ell-3}}\bigr)\ell(2\ell-1) B_{2\ell-2}}
\end{equation*}
for $\ell>m\ge2$, where we excluded the case $\lim_{s\to0^+}B_{1}(s)=B_1=-\frac{1}{2}\ne0$, as well as
\begin{equation*}
\lim_{s\to\frac12}\frac{B_{2m}(s) -B_{2m}}{B_{2\ell}(s)-B_{2\ell}}
=\frac{\bigl(1-\frac1{2^{2m}}\bigr)B_{2m}}{\bigl(1-\frac1{2^{2\ell}}\bigr) B_{2\ell}}
=\lim_{s\to0^+}\frac{B_{2m}(s) -B_{2m}\bigl(\frac12\bigr)}{B_{2\ell}(s)-B_{2\ell}\bigl(\frac12\bigr)}
\end{equation*}
for $\ell>m\in\mathbb{N}$.
Applying the decreasing property given in Corollary~\ref{T4}, we derive
\begin{equation*}
(-1)^{\ell-m}\frac{\bigl(1-\frac1{2^{2m}}\bigr)B_{2m}}{\bigl(1-\frac1{2^{2\ell}}\bigr) B_{2\ell}}
<(-1)^{\ell-m}\frac{B_{2m}(s)-B_{2m}}{B_{2\ell}(s)-B_{2\ell}}
<(-1)^{\ell-m}\frac{m(2m-1)B_{2m-2}}{\ell(2\ell-1) B_{2\ell-2}}
\end{equation*}
and
\begin{align*}
(-1)^{\ell-m}\frac{m(2m-1) \bigl(1-\frac1{2^{2m-3}}\bigr)B_{2m-2}}{\bigl(1-\frac1{2^{2\ell-3}}\bigr)\ell(2\ell-1) B_{2\ell-2}}
&<(-1)^{\ell-m}\frac{B_{2m}(s) -B_{2m}\bigl(\frac12\bigr)}{B_{2\ell}(s)-B_{2\ell}\bigl(\frac12\bigr)}\\
&<(-1)^{\ell-m}\frac{\bigl(1-\frac1{2^{2m}}\bigr)B_{2m}}{\bigl(1-\frac1{2^{2\ell}}\bigr) B_{2\ell}}
\end{align*}
for $\ell>m\ge2$ and $s\in\bigl(0,\frac12\bigr)$. By the fact that the function $f_{2\ell}(s)=(-1)^{\ell+1}\frac{B_{2\ell}(s)}{(2\ell)!}$ decreases in $s\in\bigl(0,\frac12\bigr)$ for $\ell\in\mathbb{N}$, which was proved in the proof of Theorem~\ref{T3}, we can reformulate the above two double inequalities as
\begin{multline}\label{n2infty-ineq1}
(-1)^{m}\biggl(1-\frac1{2^{2m}}\biggr)B_{2m}\frac{B_{2\ell}(s)-B_{2\ell}}{\bigl(1-\frac1{2^{2\ell}}\bigr) B_{2\ell}}
<(-1)^{m}[B_{2m}(s)-B_{2m}]\\
<(-1)^{m}m(2m-1)B_{2m-2}\frac{B_{2\ell}(s)-B_{2\ell}}{\ell(2\ell-1) B_{2\ell-2}}
\end{multline}
and
\begin{multline}\label{n2infty-ineq2}
(-1)^{m+1}m(2m-1) \biggl(1-\frac1{2^{2m-3}}\biggr)B_{2m-2}\frac{B_{2\ell}(s)-B_{2\ell}\bigl(\frac12\bigr)}{\bigl(1-\frac1{2^{2\ell-3}}\bigr)\ell(2\ell-1) B_{2\ell-2}}\\
<(-1)^{m+1}\biggl[B_{2m}(s) -B_{2m}\biggl(\frac12\biggr)\biggr]
<\biggl(1-\frac1{2^{2m}}\biggr)|B_{2m}|\frac{B_{2\ell}(s)-B_{2\ell}\bigl(\frac12\bigr)}{\bigl(1-\frac1{2^{2\ell}}\bigr) B_{2\ell}}
\end{multline}
for $\ell>m\ge2$ and $s\in\bigl(0,\frac12\bigr)$.
\par
Making use of the asymptotic approximations~\eqref{sim-even-even1} and~\eqref{sim-even-even2}, we deduce
\begin{equation*}
\frac{B_{2\ell}(s)-B_{2\ell}}{\bigl(1-\frac1{2^{2\ell}}\bigr) B_{2\ell}}
=\frac{1}{1-\frac1{2^{2\ell}}}\biggl[\frac{B_{2\ell}(s)}{B_{2\ell}}-1\biggr]\sim\cos(2\pi s)-1
\end{equation*}
and
\begin{equation*}
\frac{B_{2\ell}(s)-B_{2\ell}\bigl(\frac12\bigr)}{\bigl(1-\frac1{2^{2\ell}}\bigr) B_{2\ell}}
=\frac{1}{1-\frac1{2^{2\ell}}}\biggl[\frac{B_{2\ell}(s)}{B_{2\ell}}-\frac{B_{2\ell}\bigl(\frac12\bigr)}{B_{2\ell}}\biggr]
\sim\cos(2\pi s)+1
\end{equation*}
as $\ell\to\infty$. Employing the asymptotic approximations in~\eqref{p594Entry24.11.5NIST-HB-2010} leads to
\begin{equation*}
\frac{B_{2\ell}(s)-B_{2\ell}}{\ell(2\ell-1) B_{2\ell-2}}
\sim\frac{1}{\ell(2\ell-1)} \frac{(-1)^{\lfloor \ell\rfloor-1}\frac{2(2\ell)!}{(2\pi)^{2\ell}}[\cos(2\pi s)-1]} {(-1)^{\lfloor \ell-1\rfloor-1}\frac{2(2\ell-2)!}{(2\pi)^{2\ell-2}}}
=\frac{1-\cos(2\pi s)}{2\pi^2}
\end{equation*}
and
\begin{equation*}
\frac{B_{2\ell}(s)-B_{2\ell}\bigl(\frac12\bigr)}{\bigl(1-\frac1{2^{2\ell-3}}\bigr)\ell(2\ell-1)B_{2\ell-2}}
\sim\frac{(-1)^{\lfloor \ell\rfloor-1}\frac{2(2\ell)!}{(2\pi)^{2\ell}}[\cos(2\pi s)+1]}{\ell(2\ell-1)(-1)^{\lfloor \ell-1\rfloor-1}\frac{2(2\ell-2)!}{(2\pi)^{2\ell-2}}}
=-\frac{1+\cos(2\pi s)}{2\pi^2}
\end{equation*}
as $\ell\to\infty$. Accordingly, taking $\ell\to\infty$ in~\eqref{n2infty-ineq1} and~\eqref{n2infty-ineq2} and using the above four asymptotic approximations figure out
\begin{multline}\label{n2ity-ineq1}
(-1)^{m}\biggl(1-\frac1{2^{2m}}\biggr)B_{2m}[\cos(2\pi s)-1]
<(-1)^{m}[B_{2m}(s)-B_{2m}]\\
<(-1)^{m}m(2m-1)B_{2m-2}\frac{1-\cos(2\pi s)}{2\pi^2}
\end{multline}
and
\begin{multline}\label{n2ity-ineq2}
(-1)^{m}m(2m-1) \biggl(1-\frac1{2^{2m-3}}\biggr)B_{2m-2}\frac{1+\cos(2\pi s)}{2\pi^2}\\
<(-1)^{m+1}\biggl[B_{2m}(s) -B_{2m}\biggl(\frac12\biggr)\biggr]
<\biggl(1-\frac1{2^{2m}}\biggr)|B_{2m}|[\cos(2\pi s)+1]
\end{multline}
for $m\ge2$ on $\bigl(0,\frac12\bigr)$. The right-hand side of the double inequality~\eqref{n2ity-ineq1} can be simplified as the inequality~\eqref{B2n-b5b} for $\ell\ge2$ in $s\in\bigl(0,\frac12\bigr)$, while the left-hand side of the double inequality~\eqref{n2ity-ineq1} can be rearranged as the right-hand side of the double inequality~\eqref{B2n-b5a} for $\ell\ge2$ in $s\in\bigl(0,\frac12\bigr)$. In view of the identity~\eqref{Bernoulli-polyn-numb}, the double inequality~\eqref{n2ity-ineq2} can be reformulated as the double inequality~\eqref{B2n-b5a} for $\ell\ge2$ in $s\in\bigl(0,\frac12\bigr)$.
\par
Moreover, when $\ell=1$, the double inequality~\eqref{B2n-b5a} and the reversed version of the inequality~\eqref{B2n-b5b} become
\begin{equation*}
-\frac{1+\cos(2\pi s)} {2\pi^2} -\frac1{12}
<s^2-s+\frac{1}{6}
<\frac{1+3\cos(2\pi s)}{24}
\end{equation*}
and
\begin{equation*}
s^2-s+\frac{1}{6}<\frac{\cos (2 \pi s)}{2 \pi^2}+\frac{1}{6}-\frac{1}{2 \pi^2}.
\end{equation*}
These inequalities can be elementarily verified for $s\in(0,1)\setminus\bigl\{\frac{1}{2}\bigr\}$.
\par
On the interval $\bigl(\frac12,1\bigr)$, all proofs are straightforward repetitions. The required proof is complete.
\end{proof}

\begin{rem}
The lower bounds in~\eqref{B2n-b5a} and~\eqref{B2n-b5b} are not included each other.
\end{rem}

\begin{rem}
Since 
\begin{equation*}
\cos(2\pi s)-\frac{1+\bigl(2^{2\ell}-1\bigr) \cos(2\pi s)}{2^{2\ell}} =-\frac{1-\cos(2\pi s)}{2^{2\ell}}<0
\end{equation*}
and
\begin{equation*}
\biggl(1-\frac1{2^{2\ell-1}}\biggr) \cos(2\pi s) -\frac{1+\bigl(2^{2\ell}-1\bigr) \cos(2\pi s)}{2^{2\ell}} =-\frac{1+\cos(2\pi s)}{2^{2\ell}}<0,
\end{equation*}
the upper bound given in~\eqref{B2n-b5a} is weaker than the corresponding ones in~\eqref{B2n-b3a} and~\eqref{B2n-b3b}.
\end{rem}

\begin{rem}
The double inequality~\eqref{n2ity-ineq1} bounds the differences $B_{2m}(s)-B_{2m}$ for $m\in\mathbb{N}$ on $(0,1)$. We note that the differences $B_{2m}(s)-B_{2m}$ for $m\in\mathbb{N}$ were investigated in the paper~\cite{era-905.tex}.
\end{rem}

\begin{rem}
Taking $s=\frac{1}{2}$ in~\eqref{B2n-b5b}, using the identity~\eqref{Bernoulli-polyn-numb}, and replacing $\ell$ by $\ell+1$ reveal
\begin{equation}\label{reveal-ratio-bernou-bound}
\biggl|\frac{B_{2\ell+2}}{B_{2\ell}}\biggr|<\frac{2^{2\ell+1}}{2^{2\ell+2}-1}\frac{(\ell+1)(2\ell+1)}{\pi^2}, \quad \ell\in\mathbb{N}.
\end{equation}
This inequality is weaker than the right-hand side of the double inequality~\eqref{ineq-Bernou-equiv}, but it is better than the upper bound of the double inequality~\eqref{(29)(30)-ineq-ratio-Bern}.
\end{rem}

\begin{thm}
For $\ell\in\mathbb{N}$, the double inequalities
\begin{equation*}
\frac{6s^2-6s+1}{s(2s-1) (s-1)}\le\frac{(2\ell+1) B_{2\ell}(s)}{B_{2\ell+1}(s)}<2\pi \cot (2\pi s)
\end{equation*}
and
\begin{equation*}
\frac{6s^2-6s+1}{3(1-2s)} \le-\frac{B_{2\ell}(s)}{\ell B_{2\ell-1}(s)}
<\frac{\cot(2\pi s)}{\pi}
\end{equation*}
validate in $s\in\bigl(0,\frac12\bigr)$ and reverse in $s\in\bigl(\frac12,1\bigr)$.
\end{thm}

\begin{proof}
This follows from combining the monotonicity of the sequence $\frac{(2\ell+1)B_{2\ell}(s)}{B_{2\ell+1}(s)}$ in $\ell\in\mathbb{N}_0$ and the sequence $\frac{B_{2\ell}(s)}{\ell B_{2\ell-1}(s)}$ in $\ell\in\mathbb{N}$ for fixed $s\in\bigl(0,\frac12\bigr)\cup\bigl(\frac12,1\bigr)$ with the limits~\eqref{bernoulli-ratio-n2infty} and~\eqref{limit-Ber2n=2n-1-cot} in Theorems~\ref{T5} and~\ref{T6}.
\end{proof}

\begin{thm}
The sequence $\frac{|B_{2\ell}|}{(2\ell)!}$ is logarithmically convex in $\ell\in\mathbb{N}$, while the sequence $\frac{|B_{2\ell}(\frac{1}{2})|}{(2\ell)!}$ is logarithmically concave in $\ell\in\mathbb{N}$. Consequently, the sequence $\zeta(2\ell)$ is logarithmically convex in $\ell\in\mathbb{N}$, while the sequence $\bigl(1-\frac1{2^{2\ell-1}}\bigr)\zeta(2\ell)$ is logarithmically concave in $\ell\in\mathbb{N}$.
\end{thm}

\begin{proof}
The decreasing property of the sequence $\frac{B_{2\ell}(s)}{\ell B_{2\ell-1}(s)}$ in $\ell\in\mathbb{N}$ for fixed $s\in\bigl(0,\frac12\bigr)$ and the limit~\eqref{limit-Ber2n=2n-1-cot} in Theorem~\ref{T6} imply that, for $\ell\in\mathbb{N}$ and $s\in\bigl(0,\frac12\bigr)$,
\begin{equation*}
\lim_{s\to0^+}\frac{s B_{2\ell}(s)}{\ell B_{2\ell-1}(s)}
\ge \lim_{s\to0^+}\frac{s B_{2\ell+2}(s)}{(\ell+1) B_{2\ell+1}(s)}
\ge -\lim_{s\to0^+}\frac{s\cot (2\pi s)}{\pi}
\end{equation*}
and
\begin{equation*}
\lim_{s\to(\frac{1}{2})^-}\frac{\bigl(s-\frac{1}{2}\bigr) B_{2\ell}(s)}{\ell B_{2\ell-1}(s)}
\le \lim_{s\to(\frac{1}{2})^-}\frac{\bigl(s-\frac{1}{2}\bigr) B_{2\ell+2}(s)}{(\ell+1) B_{2\ell+1}(s)}
\le -\lim_{s\to(\frac{1}{2})^-}\frac{\bigl(s-\frac{1}{2}\bigr)\cot(2\pi s)}{\pi}.
\end{equation*}
These inequalities are equivalent to, by the L'H\^opital rule and the formula~\eqref{Bernou-polyn-deriv},
\begin{equation}\label{bernoulli-ratio-pi-b1}
\frac{B_{2\ell}}{\ell(2\ell-1)B_{2\ell-2}}
\ge \frac{B_{2\ell+2}}{(\ell+1)(2\ell+1)B_{2\ell}}
\ge -\frac{1}{2\pi^2}
\end{equation}
and
\begin{equation}\label{bernoulli-ratio-pi-b2}
\frac{B_{2\ell}\bigl(\frac{1}{2}\bigr)}{\ell(2\ell-1)B_{2\ell-2}\bigl(\frac{1}{2}\bigr)}
\le \frac{B_{2\ell+2}\bigl(\frac{1}{2}\bigr)}{(\ell+1)(2\ell+1)B_{2\ell+1}\bigl(\frac{1}{2}\bigr)}
\le -\frac{1}{2\pi^2}
\end{equation}
for $\ell\in\mathbb{N}$. The left-hand sides in the inequalities~\eqref{bernoulli-ratio-pi-b1} and~\eqref{bernoulli-ratio-pi-b2} can be reformulated as
\begin{equation*}
\frac{|B_{2\ell}|/(2\ell)!}{|B_{2\ell-2}|/(2\ell-2)!}
\le\frac{|B_{2\ell+2}|/(2\ell+2)!}{|B_{2\ell}|/(2\ell)!}
\end{equation*}
and
\begin{equation*}
\frac{\bigl|B_{2\ell}\bigl(\frac{1}{2}\bigr)\bigr|/(2\ell)!}{\bigl|B_{2\ell-2}\bigl(\frac{1}{2}\bigr)\bigr|/(2\ell-2)!}
\ge\frac{\bigl|B_{2\ell+2}\bigl(\frac{1}{2}\bigr)\bigr|/(2\ell+2)!}{\bigl|B_{2\ell}\bigl(\frac{1}{2}\bigr)\bigr|/(2\ell)!}
\end{equation*}
for $\ell\in\mathbb{N}$. Consequently, the sequence $\frac{|B_{2\ell}|}{(2\ell)!}$ is logarithmically convex in $\ell\in\mathbb{N}$, while the sequence $\frac{|B_{2\ell}(\frac{1}{2})|}{(2\ell)!}$ is logarithmically concave in $\ell\in\mathbb{N}$.
\par
Using the relations~\eqref{Bernoul-zeta-rel} and~\eqref{Bernoulli-polyn-numb} in sequence leads to
\begin{equation*}
\frac{|B_{2\ell}|}{(2\ell)!}=\frac{2\zeta(2\ell)}{(2\pi)^{2\ell}}\quad\text{and}\quad
\frac{|B_{2\ell}(\frac{1}{2})|}{(2\ell)!}=2\biggl(1-\frac1{2^{2\ell-1}}\biggr)\frac{\zeta(2\ell)}{(2\pi)^{2\ell}}
\end{equation*}
for $\ell\in\mathbb{N}$. As a result, the sequence $\zeta(2\ell)$ is logarithmically convex in $\ell\in\mathbb{N}$, while the sequence $\bigl(1-\frac1{2^{2\ell-1}}\bigr)\zeta(2\ell)$ is logarithmically concave in $\ell\in\mathbb{N}$. The required proof is complete.
\end{proof}

\begin{rem}
In~\cite{Cerone-Dragomie-Math-Nachr-2009}, Cerone and Dragomir proved that the reciprocal $\frac{1}{\zeta(s)}$ is concave on $(1,\infty)$. 
In the paper~\cite{RCSM-D-21-00302.tex}, the following conclusions were proved:
\begin{enumerate}
\item
The sequence $\bigl|\frac{B_{2\ell+2}}{B_{2\ell}}\bigr|$ increases in $\ell\in\mathbb{N}_0$, and then the sequence $|B_{2\ell}|$ is logarithmically convex in $\ell\in\mathbb{N}_0$.
\item
For fixed $q\in\mathbb{N}$, the sequence
\begin{equation*}
\frac{\prod_{k=1}^{q}[2(\ell+1)+k]}{\prod_{k=1}^{q}(2\ell+k)}\biggl|\frac{B_{2\ell+2}}{B_{2\ell}}\biggr|
\end{equation*}
increases in $\ell\in\mathbb{N}$, and then the sequence $(2\ell+q)!\frac{|B_{2\ell}|}{(2\ell)!}$ is logarithmically convex in $\ell\in\mathbb{N}$.
\end{enumerate}
See also~\cite[Theorem~5 and Remark~1]{MIA-9509.tex}.
\par
The logarithmic concavity of the sequence $\bigl(1-\frac1{2^{2\ell-1}}\bigr)\zeta(2\ell)$ for $\ell\in\mathbb{N}$ is a special case of the fact in~\cite{Wang-JCCU-14-1998} which reads that the Dirichlet eta function $\eta(s)=\bigl(1-\frac{1}{2^{s-1}}\bigr)\zeta(s)$ is logarithmically concave in $s\in(0,\infty)$. This fact is the key to establish the double inequality~\eqref{ineq-Bernou-equiv} in~\cite{CAM-D-18-00067.tex}.
\par
Let $\alpha>0$ be a constant. In~\cite{Mon-Eta-Ratio.tex}, it was proved that both of the functions
\begin{equation*}
x\mapsto\binom{x+\alpha+q}{\alpha}\frac{\eta(x+\alpha)}{\eta(x)}, \quad q=0,1
\end{equation*}
increase from $(0,\infty)$ onto $(0,\infty)$, and then the functions $\Gamma(x+q)\eta(x)$ are both logarithmically convex in $x\in(0,\infty)$.
\par
By the way, the papers~\cite{Mon-Eta-Ratio.tex, RCSM-D-21-00302.tex} are siblings of the papers~\cite{Guo-Qi-Zeta-Riemann.tex, dema-D-22-00076.tex, HJMS1099250.tex}.
\end{rem}

\begin{rem}
From the right-hand sides of the inequalities~\eqref{bernoulli-ratio-pi-b1} and~\eqref{bernoulli-ratio-pi-b2}, we can deduce the double inequality
\begin{equation}\label{B2n+2/B2n<>YQ}
\frac{2^{2\ell}-2}{2^{2\ell+1}-1}\frac{(\ell+1)(2\ell+1)}{\pi^2}
<\biggl|\frac{B_{2\ell+2}}{B_{2\ell}}\biggr|
<\frac{(\ell+1)(2\ell+1)}{2\pi^2}, \quad \ell\in\mathbb{N}.
\end{equation}
\par
The lower bound in~\eqref{B2n+2/B2n<>YQ} is same as the corresponding one in~\eqref{ineq-Bernou-equiv} and is better than the lower bound in~\eqref{(29)(30)-ineq-ratio-Bern}, while the upper bound in~\eqref{B2n+2/B2n<>YQ} is better than the corresponding ones in~\eqref{(29)(30)-ineq-ratio-Bern} and~\eqref{reveal-ratio-bernou-bound}, but it is worse than the corresponding one in~\eqref{ineq-Bernou-equiv}.
\end{rem}

\begin{thm}
The double inequality
\begin{multline}\label{Final-satisfacty-doubl-ineq}
\frac{2^{2\ell+3}\bigl(2^{2\ell-1}-1\bigr)} {\bigl(2^{2\ell+2}-1\bigr)\bigl(2^{2\ell+1}-1\bigr)} \frac{(\ell+1)(2\ell+1)}{\pi^2}
\le \biggl|\frac{B_{2\ell+2}}{B_{2\ell}}\biggr|\\
\le \frac{2^{4\ell+2}}{\bigl(2^{2\ell+2}-1\bigr)\bigl(2^{2\ell+1}+1\bigr)} \frac{(\ell+1)(2\ell+1)}{\pi^2}
\end{multline}
validates for $\ell\in\mathbb{N}$. The left-hand side of the double inequality~\eqref{Final-satisfacty-doubl-ineq} also validates for the case $\ell=0$.
\end{thm}

\begin{proof}
The inequality~\eqref{f2m-2n-r-2n} can be reformulated as
\begin{equation*}
\frac{f_{2\ell}(s)}{f_{2\ell-1}(s)}\ge \frac{f_{2m}(s)}{f_{2m-1}(s)}-\frac{f_{2m}(r_{2\ell})}{f_{2m-1}(s)}>\frac{f_{2m}(s)}{f_{2m-1}(s)}
\end{equation*}
for $0<s<\frac{1}{2}$ and $m,\ell\in\mathbb{N}$ such that $m<\ell$. By the definitions in~\eqref{f2n=f2n+1}, we can transfer this inequality to
\begin{equation*}
-\frac{B_{2\ell}(s)}{\ell B_{2\ell-1}(s)}\ge -\frac{B_{2m}(s)}{m B_{2m-1}(s)}+\frac{B_{2m}(r_{2\ell})}{m}\frac{1}{B_{2m-1}(s)}
>-\frac{B_{2m}(s)}{m B_{2m-1}(s)}
\end{equation*}
for $0<s<\frac{1}{2}$ and $m,\ell\in\mathbb{N}$ such that $m<\ell$. Taking $\ell\to\infty$, using the limit~\eqref{limit-Ber2n=2n-1-cot}, and employing the fact $r_{2\ell}\to\frac{1}{4}$ as $\ell\to\infty$ in Lemma~\ref{Ostrowski-1960-Murch-lem}, we arrive at
\begin{equation}\label{pass-ineq-bernoul-ratio}
\frac{\cot(2\pi s)}{\pi}\ge -\frac{B_{2m}(s)}{m B_{2m-1}(s)}+\frac{B_{2m}\bigl(\frac{1}{4}\bigr)}{m}\frac{1}{B_{2m-1}(s)}
>-\frac{B_{2m}(s)}{m B_{2m-1}(s)}
\end{equation}
for $m\in\mathbb{N}$ and $0<s<\frac{1}{2}$. Multiplying all sides of the double inequality~\eqref{pass-ineq-bernoul-ratio} by $s>0$, letting $s\to0^+$, applying the L'H\^opital rule to the limit $\lim_{s\to0^+}\frac{s}{B_{2m-1}(s)}$ for $m\ge2$, and using the formula~\eqref{Bernou-polyn-deriv}, we acquire
\begin{equation}\label{suppl-ineq-bernou}
\frac{1}{2\pi^2}\ge\frac{1}{m(2m-1)}\biggl[-\frac{B_{2m}}{B_{2m-2}}+\frac{B_{2m}\bigl(\frac{1}{4}\bigr)}{B_{2m-2}}\biggr]
>-\frac{B_{2m}}{m(2m-1) B_{2m-2}}
\end{equation}
for $m\ge2$. Applying the identity~\eqref{abramp806-nistp590} to the case $\ell=2m$ for $m\in\mathbb{N}$ in~\eqref{suppl-ineq-bernou} and simplifying result in
\begin{equation*}
\frac{m(2m-1)}{2\pi^2}\ge\frac{2^{4m-1}+2^{2m-1}-1}{2^{4m-1}}\biggl|\frac{B_{2m}}{B_{2m-2}}\biggr|
>\biggl|\frac{B_{2m}}{B_{2m-2}}\biggr|, \quad m\ge2.
\end{equation*}
The left-hand side of this double inequality is just the right-hand side of the double inequality~\eqref{Final-satisfacty-doubl-ineq}.
\par
Multiplying all sides of the double inequality~\eqref{pass-ineq-bernoul-ratio} by $\frac{1}{2}-s>0$, letting $s\to\bigl(\frac{1}{2}\bigr)^-$, applying the L'H\^opital rule to the limit $\lim_{s\to(\frac{1}{2})^-}\frac{\frac{1}{2}-s}{B_{2m-1}(s)}$ for $m\in\mathbb{N}$, and using the formula~\eqref{Bernou-polyn-deriv}, we reveal
\begin{equation}\label{lim1/2=ineq-bernoul}
-\frac{m(2m-1)}{2\pi^2}\ge \frac{B_{2m}\bigl(\frac{1}{2}\bigr)-B_{2m}\bigl(\frac{1}{4}\bigr)}{B_{2m-2}\bigl(\frac{1}{2}\bigr)}
>\frac{B_{2m}\bigl(\frac{1}{2}\bigr)}{B_{2m-2}\bigl(\frac{1}{2}\bigr)}
\end{equation}
for $m\in\mathbb{N}$. Making use of the identities~\eqref{Bernoulli-polyn-numb} and~\eqref{abramp806-nistp590} on the left-hand side of the double inequality~\eqref{lim1/2=ineq-bernoul} demonstrates the left-hand side of the double inequality~\eqref{Final-satisfacty-doubl-ineq}. The required proof is complete.
\end{proof}

\begin{rem}
A direct computation shows
\begin{equation*}
\frac{2^{2\ell+3}\bigl(2^{2\ell-1}-1\bigr)} {\bigl(2^{\ell+1}-1\bigr) \bigl(2^{\ell+1}+1\bigr) \bigl(2^{2\ell+1}-1\bigr)} -\frac{2^{2\ell}-2} {2^{2\ell+1}-1}
=\frac{4^\ell-2}{\bigl(2^{2\ell+2}-1\bigr) \bigl(2^{2 \ell+1}-1\bigr)}>0
\end{equation*}
for $\ell\in\mathbb{N}$. Consequently, the lower bound in the double inequality~\eqref{Final-satisfacty-doubl-ineq} is better than the lower bounds in the double inequalities~\eqref{(29)(30)-ineq-ratio-Bern}, \eqref{ineq-Bernou-equiv}, and~\eqref{B2n+2/B2n<>YQ}.
\par
The upper bound in the double inequality~\eqref{Final-satisfacty-doubl-ineq} is better than the upper bounds in~\eqref{(29)(30)-ineq-ratio-Bern}, \eqref{reveal-ratio-bernou-bound}, and~\eqref{B2n+2/B2n<>YQ}, but it is worse than the upper bound in~\eqref{ineq-Bernou-equiv}.
\end{rem}

\begin{rem}
This paper is a slightly revised version of the arXiv preprints~\cite{arxiv.2405.05280v2, arxiv.2405.05280v1}.
\end{rem}

\section{Conclusions}
It is well known that the Bernoulli numbers $B_{2\ell}$ and the Bernoulli polynomials $B_\ell(s)$ are fundamental objects in mathematics and the mathematical sciences. Their equalities, identities, inequalities, explicit and closed-form formulas, recurrence relations, connections with other special numbers and polynomials such as the Euler and Stirling numbers and polynomials, integral representations, determinantal expressions, generalizations, analogues, and related topics have been extensively and insightfully investigated by numerous mathematicians over the past centuries. In particular, regarding inequalities for the Bernoulli numbers $B_{2\ell}$, we summarize them as follows:
\begin{enumerate}
\item
Inequalities that bound the Bernoulli numbers $B_{2\ell}$; see the brief review in the first section of~\cite{CAM-D-18-00067.tex} and the papers~\cite{era-905.tex, Filomat1193.tex, RCSM-D-21-00302.tex}.
\item
Inequalities that bound the ratios $\frac{B_{2\ell+2}}{B_{2\ell}}$ of the Bernoulli numbers $B_{2\ell}$; see the initial paper~\cite{CAM-D-18-00067.tex} and subsequent works~\cite{MIA-9509.tex, Filomat1193.tex, Rem-Bern-Ratio-Ineq.tex, Yang-Tian-JCAM-2020, Zhu-RACSAM-2020}.
\end{enumerate}
\par
In this work, we established the monotonicity of the four ratios
\begin{equation}\label{four-ratio-bernoulli-polyn}
\frac{B_{2\ell-1}(s)}{B_{2\ell+1}(s)}, \quad 
\frac{B_{2\ell}(s)}{B_{2\ell+1}(s)}, \quad 
\frac{B_{2m}(s)}{B_{2\ell}(s)}, \quad 
\frac{B_{2\ell}(s)}{B_{2\ell-1}(s)}
\end{equation}
of the Bernoulli polynomials $B_\ell(s)$, derived several new inequalities, and recovered a number of known inequalities for the Bernoulli polynomials $B_\ell(s)$, the Bernoulli numbers $B_{2\ell}$, and their ratios such as $\frac{B_{2\ell+2}}{B_{2\ell}}$.
\par
Compared with the interesting double inequalities in~\cite{CAM-D-18-00067.tex, Rem-Bern-Ratio-Ineq.tex, RCSM-D-21-00302.tex, Yang-Tian-JCAM-2020, Zhu-RACSAM-2020} for bounding the ratios $\frac{B_{2\ell+2}}{B_{2\ell}}$ of the Bernoulli numbers $B_{2\ell}$, the monotonic properties of the ratios in~\eqref{four-ratio-bernoulli-polyn} for the Bernoulli polynomials $B_\ell(s)$ are more interesting and creative.
\par
The origin of this paper is~\cite[Proposition~1]{Bernouli-No-Tail.tex}, namely Qi's problem and its solutions. This demonstrates that new and curious problems may arise from the study of normalized remainders of Maclaurin power series expansions of analytic functions. The concept of normalized remainders was initially introduced by Qi and has been reviewed and surveyed in~\cite{GM-VG-QiF-Ch.tex}. Consequently, this new notion of normalized remainders is undoubtedly significant, and its study is expected to continue developing in the future.

\section{Declarations}

\paragraph{\bf Authors' Contributions}
All authors contributed equally to the manuscript and read and approved the final manuscript. All authors reviewed the manuscript.

\paragraph{\bf Funding}
The second author was partially supported by the Natural Science Foundation of Inner Mongolia Autonomous Region (Grant No.~2025QN01041) and by the Youth Project of Hulunbuir City for Basic Research and Applied Basic Research (Grant No.~GH2024020).

\paragraph{\bf Institutional Review Board Statement}
Not applicable.

\paragraph{\bf Informed Consent Statement}
Not applicable.

\paragraph{\bf Ethical Approval}
The conducted research is not related to either human or animal use.

\paragraph{\bf Availability of Data and Material}
Data sharing is not applicable to this article as no new data were created or analyzed in this study.

\paragraph{\bf Acknowledgements}
Not applicable.

\paragraph{\bf Competing Interests}
The authors declare that they have no conflict of competing interests.

\paragraph{\bf Use of AI tools declaration}
The authors declare they have not used Artificial Intelligence (AI) tools in the creation of this article.

\end{document}